\documentclass[twoside,11pt]{article}

\usepackage{blindtext}

\usepackage{jmlr2e}

\usepackage{amsmath}

\usepackage{subcaption}
\usepackage{nicefrac}

\usepackage{tikz}
\usetikzlibrary{external}
\usepackage{tikz-cd}
\usepackage{pgfplots}
\usepackage{pgfplotstable}
\usetikzlibrary{3d,calc}
\usetikzlibrary{shapes, spy, scopes, matrix, datavisualization.formats.functions, arrows, datavisualization, plothandlers, plotmarks, svg.path, positioning}
\usetikzlibrary{external}
\tikzset{external/only named=true}
\usetikzlibrary{spy}
\usetikzlibrary{svg.path}
\usetikzlibrary{positioning}
\pgfplotsset{compat=newest}
\pgfplotsset{plot coordinates/math parser=false}
\pgfplotsset{scaled y ticks = false, tick label style={/pgf/number format/fixed}}
\pgfplotsset{scaled x ticks = false, tick label style={/pgf/number format/fixed}}

\usepackage[capitalize,nameinlink]{cleveref}

\DeclareMathOperator*{\argmin}{\arg\!\min}
\newtheorem{observation}{Observation}
\newtheorem{assumption}[theorem]{Assumption}

\usepackage{xcolor}
\definecolor{steelblue}{rgb}{0.27,0.51,0.71}

\def\N{\mathbb N}

\def\R{\mathbb R}

\def\R{\mathbb{R}}
\def\N{\mathbb{N}}

\def\cA{\mathcal{A}}

\def\cD{\mathcal{D}}

\def\cG{\mathcal{G}}
\def\cH{\mathcal{H}}
\def\cI{\mathcal{I}}

\def\cL{\mathcal{L}}
\def\cM{\mathcal{M}}

\def\cQ{\mathcal{Q}}

\def\cU{\mathcal{U}}
\def\cV{\mathcal{V}}

\def\ff{\boldsymbol{f}}

\def\uu{\boldsymbol{u}}
\def\vv{\boldsymbol{v}}
\def\ww{\boldsymbol{w}}
\def\xx{\boldsymbol{x}}
\def\yy{\boldsymbol{y}}
\def\zz{\boldsymbol{z}}

\def\II{\boldsymbol{I}}

\def\KK{\boldsymbol{K}}

\def\11{\mathbf{1}}

\def\00{\boldsymbol{0}}

\def\aalpha{\boldsymbol{\alpha}}
\def\bbeta{\boldsymbol{\beta}}

\def\zzeta{\boldsymbol{\zeta}}

\def\xxi{\boldsymbol{\xi}}

\def\oomega{\boldsymbol{\omega}}

\def\frakf{\mathfrak{f}}

\def\transpose{\text{T}}

\def\spn{\operatorname{span}}

\usepackage{lastpage}

\ShortHeadings{Kernel-based Operator Learning}{R.~Kempf}
\firstpageno{1}

\begin{document}

\title{Kernel-based Operator Learning: Error Analysis, Budget Allocation, and a Physics-Informed Extension}

\author{\name R\"udiger Kempf \email ruediger.kempf@uni-bayreuth.de \\
       \addr Applied and Numerical Analysis\\
       University of Bayreuth\\
       95440 Bayreuth, Germany}

\editor{}

\maketitle

\begin{abstract}
We study kernel-based operator learning in a two-stage sampling framework, where an offline kernel regression operator learns a discretized representation of the target operator from input-output pairs and an online kernel reconstruction operator recovers the output function from predicted observations.

Our main theoretical contribution is an explicit budget allocation condition relating the number $N$ of training pairs, the number $n$ of input observations, and the output resolution $m$. The condition is derived from a coupled error analysis that interprets the surrogate as a reconstruction from approximate data. This yields a decomposition of the total error into reconstruction and learning contributions that can be analyzed independently. As a consequence, we obtain quantitative scaling laws describing how $N$, $n$, and $m$ must be coupled to guarantee convergence and to balance offline learning and online reconstruction errors. The resulting estimates extend previous analyses of kernel-based operator learning.

We further introduce a physics-informed extension that incorporates knowledge of the underlying PDE at evaluation time. Rather than encoding constraints directly into the kernel, we augment the online reconstruction step by penalizing PDE residuals at collocation points. The method requires no retraining for new inputs. Numerical experiments illustrate the theoretical findings and demonstrate the effectiveness of the proposed physics-informed reconstruction strategy.
\end{abstract}


\section{Introduction}\label{sec:Introduction}

Operator learning has emerged as a central problem in 
scientific machine learning: given a finite collection 
of input-output pairs generated by an unknown operator 
$\cG: \cU \to \cV$ between function spaces, the goal 
is to construct a surrogate $\cA \approx \cG$ that 
generalizes to unseen inputs. The canonical motivation 
is the solution operator of a PDE, where 
$\cG$ maps a coefficient or forcing function to the 
corresponding solution, and one wishes to replace 
expensive repeated PDE solves with a fast surrogate 
at deployment time. 

The dominant paradigm for operator learning is 
neural-network-based. Architectures such as DeepONet 
\citep{Lu2021,Lu2022,Wang2021,Wang2022} and the Fourier Neural Operator 
\citep{Li2021} have achieved impressive empirical 
performance on a range of benchmarks. However, 
neural-network approaches offer limited theoretical 
guarantees: convergence rates, sample complexity, 
and the influence of design parameters such as 
training set size and output resolution are 
difficult to characterize rigorously. Kernel methods, 
by contrast, come with a well-developed approximation 
theory \citep{Wendland2004,Schaback2006,Owhadi2019,Schoelkopf2001}, deterministic error bounds \citep{Narcowich2005,Arcangeli2012,Duchon1978}, and transparent 
dependence on all design parameters, such as the correlation length of the kernel \citep{Wendland2005,Sun2026}. While classical 
kernel methods are well understood in the 
finite-dimensional function approximation setting, 
their extension to operator learning is more recent and mainly of numerical nature \citep{Kadri2016,Owhadi2023,Batlle2024,Sharma2026,Nelson2024}. This paper contributes to the 
rigorous, approximation-theoretic foundation of 
kernel-based operator learning.

The framework we study builds on the two-stage 
approach of \citet{Batlle2024, Sharma2026}, in which an 
\emph{offline} kernel regression operator $A_{\mathrm{off}}$ 
learns a discretized version of $\cG$ from training 
pairs, and an \emph{online} kernel reconstruction 
operator $A_{\mathrm{on}}$ recovers the output function from 
the predicted observations. Related kernel-based 
approaches include \citet{Long2024}, who use a 
three-step framework to first learn the PDE form 
from data and then solve it with a kernel solver, 
and \citet{Mora2025}, who propose a hybrid GP/NN 
framework approximating the operator's associated 
bilinear form. Our approach differs from both: 
we work directly within the two-stage sampling 
structure of \citet{Batlle2024} and focus on a 
self-contained, interpretable error analysis.
 
Our main theoretical contribution is an explicit budget allocation condition relating the number $N$ of training pairs, the number $n$ of input observations, and the output resolution $m$. The condition is derived from a new coupled error analysis of kernel-based operator learning. By interpreting the surrogate as a reconstruction from approximate data, we decompose the total error into reconstruction and learning components and analyze them separately. This yields quantitative scaling laws that determine how training effort must increase as the online discretization is refined, extending the analyses of \citet{Batlle2024,Sharma2026}.

Finally, we augment the framework to incorporate knowledge 
of the underlying PDE at evaluation time. In the kernel 
setting, one classical approach encodes physical constraints 
directly into the kernel $K_\cV$ used for $ A_{\mathrm{on}} $: constructing $K_\cV$ 
whose native space is contained in a constrained function 
space, such as divergence-free or curl-free vector fields 
\citep{Narcowich1994,Fuselier2008,Narcowich2007}, ensures 
every surrogate output satisfies the constraint by 
construction. This hard-constraint approach was recently 
used for operator learning in \citet{Sharma2026}. However, 
it requires the physical constraint to be independent of 
the parameter $u$ and analytically identifiable at kernel 
construction time, ruling out parameter-dependent operators 
such as $\cL_u = -\mathrm{div}(e^u \nabla \cdot)$.

We therefore follow a soft-constraint approach in the 
spirit of kernel collocation \citep{Kansa1990, Schaback2009, 
Fasshauer2014}: we penalize the PDE residual  at a finite set of collocation points
by augmenting the Tikhonov functional 
with a weighted residual term. Unlike classical kernel 
collocation, which solves a PDE from scratch, our 
method augments an already-trained operator learning 
surrogate at evaluation time, using the collocation 
penalty only in the online reconstruction step $A_{\mathrm{on}}$. 
No PDE solver is invoked, and the method adapts to new 
parameter inputs $u$ without retraining, in contrast to 
the training-time physics-informed approach of 
\citet{Wang2021}. The resulting reconstruction operator $A_{\mathrm{on}}^{\mathrm{PDE}}$ 
admits a closed-form representer theorem via the framework 
of \citet{Micchelli2005}. We also analyse the numerical costs
for the standard and the physics-informed learning method.

The rest of the paper is structured as follows. 
\cref{sec:Setting} introduces the operator learning 
framework and its generalizations over \citet{Batlle2024}, 
with particular attention to the existence of the 
discretized operator $g$. \cref{sec:KernelLearning} 
recalls the relevant kernel approximation theory, 
including sampling inequalities. The error analysis 
is developed in \cref{sec:ErrorAnalysis}, covering 
both interpolation and regularized least-squares data 
generation, and yields the budget allocation condition. 
The physics-informed extension is treated in 
\cref{sec:PhysicsInformed}. Numerical experiments 
validating the theory are presented in 
\cref{sec:Numerics}, and \cref{sec:Conclusion} 
gives a brief conclusion.


\section{The Two-Stage Operator Learning Framework}\label{sec:Setting}

We start with giving the precise setting and fixing notation used throughout. Compared to \citet{Batlle2024} we study operator learning where input and output functions are vector fields, similar to \citet{Sharma2026}. Furthermore, we discuss the connection of the fill distance of the sampled input functions and the fill distance of the data we actually use for learning the surrogate. This also yields first results on the existence of the discretized operator $ g $.

\subsection{The Setting}
Let $ \cU $ and $ \cV $ be Banach spaces and let $ \cG: \cU \to \cV $ be an operator.
The goal of operator learning is to learn $ \cG $ from given $ N $-many input/output pairs, i.e., to find a surrogate $ \cA: \cU \to \cV $ for $ \cG $ using only the training data $ \{(u_i,v_i)\}_{i=1}^{N}:=  \{(u_i, \cG(u_i))\}_{i=1}^{N} \subseteq \cU \times \cV $.

In our framework, the input/output functions are also only known by discrete observations, modeled by sampling operators: Let $ \Phi = \{ \phi_1, \dots, \phi_n \} $ and $ \Psi := \{ \psi_1, \dots, \psi_m \} $ be sets of bounded linear observational mappings on $ \cU $ and $ \cV $, respectively. Define the \emph{sampling operators} $ S_{\Phi}: \cU \to \R^{n} $ and $ S_{\Psi}: \cV \to \R^{m} $ by 
\begin{align*}
 S_{\Phi}(u) = \left[\phi_1(u), \dots, \phi_n(u)\right]^{\transpose} \quad \text{and} \quad S_{\Psi}(v) = \left[\psi_1(v), \dots, \psi_m(v) \right]^{\transpose}.
\end{align*}
With this, we can formalize the operator learning task the following way:

\begin{center}
\begin{minipage}{0.9\textwidth}
{\itshape
Let $\{(u_i,v_i)\}_{i=1}^{N} \in \mathcal U \times \mathcal V$ be such that
\[
v_i = \cG(u_i), \quad 1 \leq i \leq N.
\]
and let $ S_{\Phi}: \cU \to \R^n $ and $ S_{\Psi}: \cV \to \mathbb R^m$ be sampling operators.

Then the operator learning task is to find a surrogate $ \cA $ for $\cG$ using only the data
\[
\{(S_{\Phi}(u_i), S_{\Psi}(v_i))\}_{i=1}^{N} := \{ (\uu_i, \vv_i) \}_{i=1}^{N}.
\]
}
\end{minipage}
\end{center}

Of particular interest, cf. \citet{Batlle2024}, are operators $ \cG $ that are connected to (non-linear) partial differential operators, e.g., coefficient-to-solution operators. In this case, $ \cU $ and $ \cV $ are Banach spaces
of functions $ u: \Omega \to \R^{p} $ and 
$ v: \cD \to \R^{\ell} $ respectively, where $ \Omega \subseteq \R^k $ and 
$ \cD \subseteq \R^d $ are non-empty domains. $ \Omega $ and $ \cD $ and $ k $ and $ d $ do not 
necessarily need to coincide. The sampling operators $ S_{\Phi} $ and $ S_{\Psi} $ are assumed to be point evaluation operators, i.e., we assume there are sets of collocation points $ X = \{ \xx_1, \dots, \xx_n\} \subseteq \Omega $ and $ Y := \{ \yy_1, \dots, \yy_m \} \subseteq \cD $ and define 
$ S_X := S_{\Phi} $ and $ S_Y := S_{\Psi}$ by

\begin{align*}
 S_X(u) := \uu := \begin{bmatrix} u(\xx_1) \\ u(\xx_2) \\ \vdots \\ u(\xx_n) \end{bmatrix} \in \R^{np}, \quad \text{and} \quad S_Y(v) := \vv := \begin{bmatrix} v(\yy_1) \\ v(\yy_2) \\ \vdots \\ v(\yy_m) \end{bmatrix} \in \R^{m\ell},
\end{align*}
as point-evaluation mappings, cf. \citet{Sharma2026}.

\begin{figure}
\begin{center}
\begin{tikzpicture}[node distance = 2cm and 4cm, thick]%
 \node(1) {$ \cU $};
 \node(2) [right=of 1] {$ \cV $};
 \node(3) [below=of 1] {$ \R^{n p} $};
 \node(4) [below=of 2] {$ \R^{m \ell} $};
 \draw[->] (1) -- node [midway, above] {$ \cG $} (2);
 \draw[->] (3) -- node [midway, above] {$ g $} node[midway, below] {$ A_{\mathrm{off}} $} (4);
 \draw[->] (1) to [bend right] node [midway, left] {$ S_{\Phi} $} (3);
 \draw[->] (3) to [bend right] node [midway, right] {$ A_{le} $} (1);
 \draw[->] (2) to [bend right] node [midway, left] {$ S_{\Psi} $} (4);
 \draw[->] (4) to [bend right] node [midway, right] {$ A_{\mathrm{on}} $} (2);
\end{tikzpicture}
\end{center}
\caption{The fully symmetric learning approach of \citet{Batlle2024}: To learn the operator $ \cG $, employ the two sampling operators $ S_{\Phi} $ and $ S_{\Psi} $ and three reconstructions $ A_{le} $, $ A_{\mathrm{off}} $ and $ A_{\mathrm{on}} $. The discretized operator $ g $ is give as $ S_{\Psi} \circ \cG \circ A_{le} $. The learned surrogate is then given as $ \cA = A_{\mathrm{on}} \circ A_{\mathrm{off}} \circ S_{\Phi} $.}
\label{fig:DiagramOperatorLearning}
\end{figure}
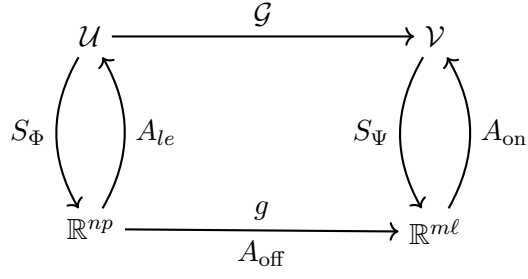

Our methodology follows a two-stage approach. We assume the 
existence of a discretized operator $g: \R^{np} \to \R^{m\ell}$ 
satisfying $g(S_X(u)) = S_Y(\cG(u))$ for all $u \in \cU$. In practice, we will restrict to a compact model calss $ \cM \subseteq \cU $, see \cref{assum:ErrorAnalysis}. In an 
\emph{offline phase}, we learn a surrogate $A_{\mathrm{off}} \approx g$ from 
the training data $\{(\uu_i, \vv_i)\}_{i=1}^{N}$ by 
kernel regression on $\R^{np}$. In the \emph{online phase}, for a 
new input $u$, a kernel reconstruction operator 
$A_{\mathrm{on}}: \R^{m\ell} \to \cV$ recovers the output function from the 
predicted observations $A_{\mathrm{off}}(S_X(u)) \approx S_Y(\cG(u))$, giving 
the surrogate
\begin{align*}
    \cA = A_{\mathrm{on}} \circ A_{\mathrm{off}} \circ S_X: \cU \to \cV.
\end{align*}
The key insight is to view $A_{\mathrm{on}}$ as a reconstruction from 
\emph{perturbed} data: the exact discretized output 
$\vv^\dagger = S_Y(\cG(u))$ is unknown, and $A_{\mathrm{off}}(S_X(u))$ is 
only an approximation to it. This yields a clean decomposition of 
the total error into a \emph{reconstruction error}, measuring how 
well $A_{\mathrm{on}}$ recovers $\cG(u)$ from exact data $\vv^\dagger$, and a 
\emph{learning error}, measuring how well $A_{\mathrm{off}}$ approximates $g$. 
The two contributions can be analyzed independently, which is the 
basis of the error analysis in \cref{sec:ErrorAnalysis}. The 
existence and properties of $g$ are discussed in 
\cref{subsec:FillDistanceAndExistenceOfG,subsec:ExistenceG}.

This differs from \citet{Batlle2024}, where an additional 
reconstruction operator $A_{le}: \R^{np} \to \cU$ is assumed to 
exist and $g$ is expressed as $S_Y \circ \cG \circ A_{le}$. 
Further differences are discussed in \cref{subsec:Problems,subsec:Solutions}.

\subsection{Observations about the Established Framework}\label{subsec:Problems}

The operator learning framework described above is 
general, but several assumptions implicit in the 
existing literature \citep{Batlle2024, Sharma2026} 
can be relaxed without sacrificing the error analysis. 
We identify three such generalizations, each of which 
broadens the applicability of the framework.

\begin{observation}[Symmetry of the formulation]\label{problem1}
The diagram in \cref{fig:DiagramOperatorLearning} is symmetric in 
$\mathcal{U}$ and $\mathcal{V}$: it assumes the simultaneous 
existence of reconstruction operators $A_{le}: \mathbb{R}^{np} 
\to \mathcal{U}$ and $A_{\mathrm{on}}: \mathbb{R}^{m\ell} \to \mathcal{V}$. 
This is unnecessarily restrictive. In our framework, $A_{le}$ 
plays no role in either the surrogate construction or the error analysis: 
the surrogate $\mathcal{A} = A_{\mathrm{on}} \circ A_{\mathrm{off}} \circ S_X$ 
requires only a reconstruction on the $\mathcal{V}$-side. Consequently, 
$\mathcal{U}$ need not carry any additional structure, it
suffices for $\mathcal{U}$ to be a Banach space supporting bounded point 
evaluations. This asymmetry also opens the door to treating 
\emph{inverse problems}, where the roles of $\mathcal{U}$ and 
$\mathcal{V}$ are exchanged.
\end{observation}

\begin{observation}[Condition 3.2 in \citet{Batlle2024}]\label{problem2}
The framework of \citet{Batlle2024} assumes that the 
training inputs satisfy $u_i = A_{le}(S_\Phi(u_i))$ 
for all $1 \leq i \leq N$, i.e., that the training 
inputs are exactly reconstructible from their 
observations. While natural in their setting, this 
assumption ties the training procedure to a specific 
reconstruction operator $A_{le}$ that must be known 
prior to data collection. Our framework avoids this 
assumption entirely, since $A_{le}$ does not appear 
in our surrogate or error analysis.
\end{observation}

\begin{observation}[Fill distance in $\mathcal{U}$]\label{problem3}
For the error analysis, it is necessary to assume that the discrete 
training inputs $\{S_X(u_i)\}_{i=1}^{N} \subseteq \R^{np}$ 
are sufficiently rich, typically expressed via a small fill distance. 
However, this condition on $\{S_X(u_i)\}$ carries limited information 
about the original functions $\{u_i\} \subseteq \mathcal{U}$. The sampling operator $S_X$ is in general not injective: 
distinct functions $u, u' \in \mathcal{U}$ may satisfy $S_X(u) = S_X(u')$, 
so fill distance in $\R^{np}$ does not imply coverage of 
$\mathcal{U}$. This persists even as the number of observations grows, as the classical example of Runge \citet{Runge1905} shows.
\end{observation}

\subsection{Proposed Framework}\label{subsec:Solutions}

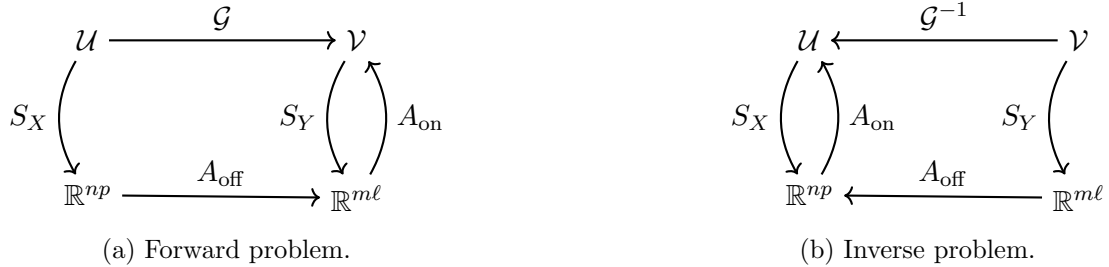
\begin{figure}[htbp]
\centering
\begin{subfigure}[t]{0.4\textwidth}
\centering
\begin{tikzpicture}[node distance = 1.5cm and 3cm, thick]
    \node(1) {$\mathcal{U}$};
    \node(2) [right=of 1] {$\mathcal{V}$};
    \node(3) [below=of 1] {$\mathbb{R}^{np}$};
    \node(4) [below=of 2] {$\mathbb{R}^{m\ell}$};
    \draw[->] (1) -- node [midway, above] {$\mathcal{G}$} (2);
    \draw[->] (3) -- node [midway, above] {$A_{\mathrm{off}}$} (4);
    \draw[->] (1) to [bend right] node [midway, left] {$S_X$} (3);
    \draw[->] (2) to [bend right] node [midway, left] {$S_Y$} (4);
    \draw[->] (4) to [bend right] node [midway, right] {$A_{\mathrm{on}}$} (2);
\end{tikzpicture}
\caption{Forward problem.}
\label{fig:DiagramForwardProblem}
\end{subfigure}
\hfill
\begin{subfigure}[t]{0.4\textwidth}
\centering
\begin{tikzpicture}[node distance = 1.5cm and 3cm, thick]
    \node(1) {$\mathcal{U}$};
    \node(2) [right=of 1] {$\mathcal{V}$};
    \node(3) [below=of 1] {$\mathbb{R}^{np}$};
    \node(4) [below=of 2] {$\mathbb{R}^{m\ell}$};
    \draw[<-] (1) -- node [midway, above] {$\mathcal{G}^{-1}$} (2);
    \draw[<-] (3) -- node [midway, above] {$A_{\mathrm{off}}$} (4);
    \draw[->] (1) to [bend right] node [midway, left] {$S_X$} (3);
    \draw[->] (3) to [bend right] node [midway, right] {$A_{\mathrm{on}}$} (1);
    \draw[->] (2) to [bend right] node [midway, left] {$S_Y$} (4);
\end{tikzpicture}
\caption{Inverse problem.}
\label{fig:DiagramInverseProblem}
\end{subfigure}
\caption{Asymmetric diagrams addressing 
\cref{problem1} and \cref{problem2}: only one 
reconstruction operator is required in each case.}
\label{fig:ProposedChanges}
\end{figure}

It turns out that \cref{problem1} and \cref{problem2} can be resolved 
simultaneously by augmenting the general setting. Instead of assuming 
the symmetric diagram in \cref{fig:DiagramOperatorLearning}, we propose 
to separate the forward and inverse problems, as summarized in 
\cref{fig:ProposedChanges}. In both cases, only a single online reconstruction 
operator is required, either on the $ \cV $ side for the forward problem, or on the 
$ \cU $ side for the inverse problem. This asymmetric formulation resolves both observations at once: 
since $A_{le}$ no longer appears, the assumption $u_i = 
A_{le}(S_\Phi(u_i))$ of \citet{Batlle2024} is not needed, and 
$\mathcal{U}$ need not carry any additional structure. In 
particular, $\mathcal{U}$ may be any Banach space of functions supporting 
bounded point evaluations.

\subsection{Fill Distance and Existence of the Discretized Operator}\label{subsec:FillDistanceAndExistenceOfG}

Addressing \cref{problem3} is more involved. We begin by formally defining the 
fill distance.

\begin{definition}[Fill distance]
Let $(V, \|\cdot\|_V)$ be a normed space and $D \subseteq V$ a bounded 
domain. Let $\Xi := \{\xxi_1, \dots, \xxi_M\} \subseteq D$ be a discrete 
set of points. Then the \emph{fill distance} $h_{\Xi, D}$ of $\Xi$ in 
$D$ is defined as
\begin{align*}
    h_{\Xi, D} := \sup_{\xxi \in D} \min_{1 \leq i \leq M} 
    \|\xxi - \xxi_i\|_V.
\end{align*}
\end{definition}

The following theorem links the fill distance of the training inputs 
$\{u_i\}$ in $\mathcal{U}$ to the fill distance of their discretizations 
$\{\uu_i\}$ in $\mathbb{R}^{np}$.

\begin{theorem}\label{thrm:SolvingProblem3}
Let $\mathcal{M} \subseteq \mathcal{U}$ be a compact model class and 
let $\{u_1, \dots, u_N\} \subseteq \mathcal{M}$. Assume that $S_\Phi$ 
consists of bounded linear observational mappings and that there exists $h_0 > 0$ 
such that
\begin{align*}
    h_{\{\uu_i\}, S_\Phi(\mathcal{M})} \leq h_0.
\end{align*}
Define $\iota_\Phi: \mathcal{M} \to \mathbb{R}_{\geq 0}$ by
\begin{align*}
    \iota_\Phi(u) := \inf_{w \in \ker S_\Phi} \|u - w\|_{\mathcal{U}},
\end{align*}
and assume there exists a constant $C > 0$ such that for all $u, 
\widetilde{u} \in \mathcal{M}$,
\begin{align}\label{eq:AssumInequalityExistenceG}
    \|u - \widetilde{u}\|_{\mathcal{U}} \leq C\|S_\Phi(u) - 
    S_\Phi(\widetilde{u})\|_2 + \iota_\Phi(u) + \iota_\Phi(\widetilde{u}).
\end{align}
Then
\begin{align*}
    \sup_{u \in \mathcal{M}} \min_{1 \leq i \leq N} 
    \|u - u_i\|_{\mathcal{U}} \leq C h_{\{\uu_i\}, 
    S_\Phi(\mathcal{M})} + 2\sup_{u \in \mathcal{M}} \iota_\Phi(u)
\end{align*}
holds for all $ u \in \cM $.
\end{theorem}

\begin{proof}
Let $u \in \mathcal{M}$ be fixed. Since $\mathcal{M}$ is compact and 
$S_\Phi$ consists of continuous mappings, there exists an index $1 \leq i \leq N$ such 
that $\|S_\Phi(u) - S_\Phi(u_i)\|_2 \leq h_{\{\uu_i\}, 
S_\Phi(\mathcal{M})}$. Then \eqref{eq:AssumInequalityExistenceG} yields
\begin{align*}
    \|u - u_i\|_{\mathcal{U}} \leq C h_{\{\uu_i\}, 
    S_\Phi(\mathcal{M})} + \iota_\Phi(u) + \iota_\Phi(u_i) 
    \leq C h_{\{\uu_i\}, S_\Phi(\mathcal{M})} + 
    2\sup_{u \in \mathcal{M}} \iota_\Phi(u).
\end{align*}
Taking the minimum over $i$ and then the supremum over $u \in 
\mathcal{M}$ yields the claim.
\end{proof}

The term $\iota_\Phi$ in \cref{thrm:SolvingProblem3} quantifies the 
intrinsic non-iden\-ti\-fia\-bi\-li\-ty of the sampling operator: functions $u$ 
and $\widetilde{u}$ satisfying $S_\Phi(u) = S_\Phi(\widetilde{u})$, 
i.e., $u - \widetilde{u} \in \ker S_\Phi$, are indistinguishable from 
the observations alone. The kernel $\ker S_\Phi$ can therefore be 
interpreted as the space of \emph{invisible perturbations}. In 
particular, if $\mathcal{U}$ is infinite-dimensional, $\iota_\Phi$ 
will in general not vanish identically on $\mathcal{M}$.

In the case where $S_\Phi = S_X$ is a point evaluation operator, 
$\iota_\Phi(u) $ measures the part 
of $u$ that is not observed at the points $X$.

Inequality \eqref{eq:AssumInequalityExistenceG} is a Lipschitz stability condition for the inverse sampling map. In the injective case $ \iota_{\Phi} \equiv 0 $, the assumption reduces to $ S_{\Phi}^{-1}: S_{\Phi}(\cM) \to \cM $ is $ C $-Lipschitz, i.e., that the inverse problem of recovering $ u $ from its observations is well-posed on $ \cM $. A concrete setting in which \eqref{eq:AssumInequalityExistenceG} is satisfied is discussed in \cref{subsec:ExistenceG}.

\begin{remark}[Solution for \cref{problem3}]
\cref{thrm:SolvingProblem3} shows that the fill-distance of $\{u_i\}$ 
in $\mathcal{M}$ decomposes into two contributions: the fill distance 
of $\{\uu_i\}$ in $S_\Phi(\mathcal{M})$, which is directly 
controllable by the choice of training data, and the intrinsic 
non-identifiability term $\sup_{u \in \mathcal{M}} \iota_{\Phi}(u)$, 
which is an irreducible consequence of the non-injectivity of $S_\Phi$. 
In our error analysis, we therefore work directly with the fill distance 
of $\{\uu_i\}$ in $S_\Phi(\mathcal{M}) \subseteq \mathbb{R}^{np}$, 
treating the non-identifiability term as a separate contribution, which we ignore here.
We further note that the assumption $\iota_{\Phi} \equiv 0$ 
on $\mathcal{M}$, i.e., that $S_\Phi$ is effectively injective on the 
model class, is sufficient for the existence of a well-defined 
discretized operator $g: \mathbb{R}^{np} \to \mathbb{R}^{m\ell}$ 
satisfying $g(S_X(u)) = S_Y(\mathcal{G}(u))$ for all $u \in \mathcal{M}$.
\end{remark}

\section{Kernel-based Learning}\label{sec:KernelLearning}

We give a brief introduction to kernel-based learning of vector-valued functions, which forms the theoretical backbone of our two-stage surrogate construction. Recall that the surrogate 
$\cA = A_{\mathrm{on}} \circ A_{\mathrm{off}} \circ S_X$ involves 
two distinct kernel-based learning problems: a regression problem 
for $A_{\mathrm{off}}$ on $\R^{np}$, approximating the 
discretized operator $g$, and a reconstruction problem for 
$A_{\mathrm{on}}$ in $\cV$, recovering an output function 
from disturbed data. The present section develops the tools needed 
for both. Throughout, we assume $d, r \in \N$ and 
$D \subseteq \R^d$ is an arbitrary domain. For background on kernel methods we refer to 
\citet{Wendland2004} for the approximation theory perspective and, for the statistical learning perspective, to 
\citet{Steinwart2008} .

\subsection{Matrix-Valued Positive Definite Kernels and RKHSs}\label{subsec:MatrixValuedKernels}

\begin{definition}
We call a continuous kernel $ K: D \times D \to \R^{r \times r} $ \emph{positive (semi-)definite}, if 
\begin{enumerate}
 \item $ K $ is symmetric, i.e., $ K(\xxi, \zzeta) = K(\zzeta,\xxi)^{\transpose} $, for all $ \xxi, \zzeta \in D $, and 
 \item the quadratic form 
  \begin{align} \label{eq:QuadraticForm}
   \sum_{i,j = 1}^{M} \langle \ww_i, K(\xxi_i,\xxi_j) \ww_j \rangle_2
 \end{align}
 is non-negative for all $ M \in \N $, all pair-wise distinct $ \{ \xxi_1, \dots, \xxi_M \} \subseteq D $ and all $ \ww_1, \dots \ww_M \in \R^r $, not all being the zero vector.
\end{enumerate}
 
The kernel is called \emph{positive definite} if the quadratic form 
\eqref{eq:QuadraticForm} is equal to zero only if
$ \ww_i = 0 $ for indices $ 1 \leq i \leq N $ where $ \xxi_i \neq \xxi_j $, $ i \neq 
j $.
\end{definition}

A key feature of positive definite kernels is that they implicitly 
define a \emph{feature map} $\varphi: D \to \cH$ such that 
$K(\xxi, \zzeta) = \langle \varphi(\xxi), \varphi(\zzeta) \rangle_{\cH}$, 
allowing learning algorithms to operate in a potentially 
infinite-dimensional \emph{feature space} $\mathcal{H}$ using only 
kernel evaluations, this is known as the \emph{kernel trick} 
\citep{Steinwart2008}. The associated feature space is a Hilbert space 
of functions with a special structure, called \emph{native space} of the kernel $ K $ or a \emph{reproducing 
kernel Hilbert space} (RKHS) with kernel $K$.

\begin{definition}\label{def:RKHS}
 A Hilbert space $ (\cH, \langle \cdot, \cdot \rangle_{\cH}) $ of vector-valued functions
 $ f: D \to \R^r $ is a 
 \emph{reproducing kernel Hilbert space} if there is a kernel $ K: D \times
 D \to \R^{r \times r}$ that satisfies
 \begin{enumerate}
  \item $ K(\cdot, \xxi)\ww \in \cH $ for all $ \xxi \in D $ and $ \ww \in \R^r $,
  \item\label{def:RKHSprop2} $ \langle v, K(\cdot,\xxi)\ww \rangle_{\cH} = \langle v(\xxi),\ww \rangle_{2} $ for all $ v \in \cH $, $ \ww \in \R^r $ and $ \xxi \in D $.
 \end{enumerate}
The latter is called \emph{reproducing property} of $ K $.
\end{definition}

A particularly important class of matrix-valued kernels are 
\emph{separable kernels}, which decouple two distinct modeling 
aspects: the spatial regularity of the function, encoded by 
scalar-valued positive definite kernels $k_i: D \times D \to 
\mathbb{R}$, and the interaction between output components, 
encoded by symmetric positive semidefinite matrices $B_i \in 
\mathbb{R}^{r \times r}$. This separation makes both the 
theoretical analysis and the computational implementation 
tractable. For fixed $b \in \mathbb{N}$, we define a 
matrix-valued kernel $K: D \times D \to \mathbb{R}^{r \times r}$ by
\begin{align}\label{eq:SeparableKernel}
    K(\xxi, \zzeta) := \sum_{i=1}^{b} k_i(\xxi, \zzeta) B_i, 
    \quad \xxi, \zzeta \in D.
\end{align}
In the special case $b = 1$ and $B_1 = I$, the kernel acts 
componentwise and the associated RKHS is the tensor product 
of identical scalar RKHSs.

Of particular interest are translation-invariant kernels on $ \R^d $. A kernel $ K: \R^d \times \R^d \to \R^{r \times r} $ is called \emph{translation invariant} if there exists a function $ \Phi: \R^d \to \R^{r \times r} $ such that 
\begin{align*}
 K(\xxi, \zzeta) := \Phi(\xxi - \zzeta), \quad \xxi,\zzeta \in \R^d.
\end{align*}
In the separable setting \eqref{eq:SeparableKernel}, this corresponds to choosing translation-invariant kernels $ k_i = \Phi_i $.

For the error analysis later, we want to identify RKHSs with known function spaces. To this end, we recall the definition of Sobolev spaces of vector-valued functions: For $ 1 \leq q \leq \infty $, let $L_q(D)^r$ denote 
the standard Lebesgue space of $q$-integrable 
vector-valued functions $f: D \to \R^r$, equipped 
with the norm $\|\cdot\|_{L_q(D)^r}$.

For $ s \in \N_0 $, the Sobolev space $ W_q^s(D)^r $ is given by 
\begin{align*}
  W^s_q(D)^r := \{ f \in L_q(D)^r \; : \: \| f \|_{W^s_q(D)^r} < \infty \},
\end{align*}
where 
\begin{align*}
 \| f \|_{W^s_q(D)^r} := \left( \sum_{l = 0}^s | f |^q_{W^l_q(D)^r} \right)^{\frac{1}{q}} \ \text{, with } | f |_{W^l_q(D)^r} := \left( \sum_{| \aalpha | = l} \| D^{\aalpha} f \|_{L_q(D)^r}^q \right)^{\frac{1}{q}}.
\end{align*}
Fractional orders $ s \geq 0 $ can be defined by, e.g., interpolation \citep{Brenner2008}.

In the Hilbert space case $ q = 2 $, Sobolev spaces admit a convenient Fourier characterization. For $ s \geq 0 $, we define 
\begin{align*}
 H^s(\R^d)^r := \{ f \in L_2(\R^d)^r \; : \; (1 + \| \cdot \|^2)^{s/2} \widehat{f} \in L_2(\R^d)^r \},
\end{align*}
equipped with the inner product 
\begin{align}\label{eq:InnerProductSobolevSpace}
 \langle f,g \rangle_{H^{s}(\R^d)^r} := \int_{\R^d} (1 + \| \oomega \|^2)^s \overline{\widehat{f}(\oomega)}^{\transpose} \widehat{g}(\oomega) \ d \oomega,
\end{align}
where the Fourier transform is applied componentwise.

The connection between RKHSs of translation-invariant kernels and Sobolev spaces is well-understood in the scalar-valued case \citep{Wendland2004}.

\begin{corollary}\label{cor:ScalarValuedSobolevKernel}
 Let $ k: \R^d \times \R^d \to \R $ be a RBF associated to $ \Phi \in L_1(\R^d) $. Let $ s > d/2 $ and assume that there are constants $ c_1, c_2 > 0 $ such that
 \begin{align*}
  c_1 (1 + \| \oomega \|_2^2)^{-s} \leq \widehat{\Phi}(\oomega) \leq c_2  (1 + \| \oomega \|_2^2)^{-s}, \quad \oomega \in \R^d.
 \end{align*}
Then the RKHS $ \cH $ with kernel $ k $ coincides with $ H^s(\R^d) $, with equivalent norms.
\end{corollary}

Prominent classes of kernels satisfying the assumptions of \cref{cor:ScalarValuedSobolevKernel} are Mat\'ern kernels \citep{Matern1986} and compactly supported Wendland kernels \citep{Wendland1995, Wendland2004}.

Using the separable construction \eqref{eq:SeparableKernel}, this result extends directly to the vector-valued setting.

\begin{corollary}
 Let $ (k_i)_{1 \leq i \leq b} $ and $ (B_i)_{1 \leq i \leq b} $ be as in \eqref{eq:SeparableKernel} and assume that each $ k_i $ satisfies the assumptions of \cref{cor:ScalarValuedSobolevKernel} with the same smoothness parameter $ s > d / 2 $. Then the RKHS associated with the matrix-valued kernel $ K $ defined as in \eqref{eq:SeparableKernel} coincides with $ H^s(\R^d)^r $. Moreover, the induced norm is equivalent to the Sobolev norm induced by the inner product in \eqref{eq:InnerProductSobolevSpace}, with constants depending on $ c_1, c_2 $ and the spectral bounds of the matrices $ B_i $.
\end{corollary}

This construction shows that separable kernels provide a systematic way to lift scalar Sobolev kernels to vector-valued function spaces, while allowing for flexible modeling of correlations between output components via the matrices $ B_i $.

\begin{remark}[Restriction to bounded domains]
In many applications, functions are defined on a bounded domain \( D \subseteq \R^d \). 
Kernels for Sobolev spaces on \(D\) can be obtained by restriction. 
Let \( k : \R^d \times \R^d \to \R \) be a translation-invariant kernel whose RKHS is \( H^s(\R^d) \), and define
\[
k|_D(\xxi,\zzeta) := k(\xxi,\zzeta), \qquad \xxi,\zzeta \in D.
\]
Then the RKHS associated with \(k|_D\) is norm-equivalent to \( H^s(D) \), provided that \(D\) is, e.g., a bounded Lipschitz domain.
An analogous construction applies to matrix-valued kernels via \eqref{eq:SeparableKernel}.
\end{remark}

\subsection{Learning Methods}

We consider two standard learning approaches within the kernel-based setting. Throughout this section, we assume that $ \cH $ is a RKHS with kernel $ K $.
Given data at a finite set of training points, we seek an approximation in a finite-dimensional hypothesis space induced by the kernel.

\begin{definition}
 For a fixed set of \emph{training points} $ \Xi := \{ \xxi_1, \dots, \xxi_M \} 
 \subseteq D $ define the 
 \emph{hypothesis space} or \emph{parametric model class} $ V_M \subseteq \cH $ as
 \begin{align*}
  V_M := \left\{ \sum_{i=1}^{M} K(\cdot, \xxi_i ) \ww_i
  \; : \;  \ww_i \in \R^r , 1 \leq i \leq M \right\}.
 \end{align*}
\end{definition}

If the data is exact, interpolation is the natural choice: 
it fits the data perfectly and, as we show below, minimizes 
the $\mathcal{H}$-norm among all interpolants. If the data is disturbed, as is the 
case for $A_{\mathrm{on}}$, which receives the inexact output 
$A_{\mathrm{off}}(S_X(u)) \approx \vv^\dagger$, 
interpolation is undesirable as it overfits the noise. In this 
case, regularized least-squares introduces a bias controlled by 
$\lambda > 0$ in exchange for improved stability, embodying the 
classical \emph{bias-variance tradeoff} familiar from statistical 
learning theory \citep{Steinwart2008}.

\subsubsection{Norm-minimal Interpolation}\label{subsubsec:Interpolation}

\begin{definition}
For given training points $ \Xi := \{ \xxi_1, \dots, \xxi_M \} 
 \subseteq D $ and corresponding labels $F := \{f_1, \dots f_M\} \subseteq \R^{r} $ the \emph{interpolation problem} is given as 
 \begin{center}
  Find $ s_0 \in V_M $ such that $ s_0(\xxi_i) = f_i $ for all 
  $ 1 \leq i \leq M $.
 \end{center}
 The solution of the interpolation problem is called the \emph{(kernel-)
interpolant} to $ F $ in $ V_M $.
\end{definition}

Kernel-interpolants in RKHSs have many 
advantageous properties that follow directly from the reproducing property
in \cref{def:RKHS}.

\begin{theorem}\label{thrm:PropertiesInterpolation}
Assume that the labels $F$ are generated by $f \in \mathcal{H}$, 
i.e., $f_i = f(\xxi_i)$, and let $s_0 \in V_M$ be the interpolant 
to $F$. 
Then $ s_F $ is the unique interpolant in $ V_M $ and the orthogonal 
projection of $ f $ onto $ V_M $. In particular, 
\begin{align*}
\|s_0\|_{\mathcal{H}} \leq \|f\|_{\mathcal{H}}
\end{align*}
holds.
\end{theorem}

The interpolant $ s_0 $ defines a linear \emph{interpolation operator} $ \cI_{\Xi}: \R^{rM} \to V_M $, $ \cI_{\Xi}([f]) = s_0 $, where $ [f] = [f_1, \dots, f_M]^{\transpose} \in \R^{rM}$. If the labels are generated by a function $ f \in C(D)^r $, this can be extended to an operator $ \cI_{\Xi}: C(D)^r \to V_M $ by setting $ \cI_{\Xi}(f) = \cI_{\Xi}([f]) $.

\subsubsection{Regularized Least-Squares Regression}\label{subsubsec:PLS}

If the data is not exact interpolation is typically undesirable, as it leads to overfitting. We try to learn the surrogate via another kernel method, \emph{regularized least-squares} or \emph{ridge regression}.

\begin{definition}\label{def:PLS}
 Let $ \lambda > 0 $ be a \emph{penalization parameter}. Let $ \Xi := \{ \xxi_1, \dots, \xxi_M \} \subseteq D $ be a set of training points and $ F = \{f_1, \dots, f_M \} \subseteq \R^{r} $ be the labels. Define the functional $ J_{\lambda} : \cH \to \R $ by 
 \begin{align}\label{eq:PLSFunctional}
  J_{\lambda}(s) := \sum_{j = 1}^{M} \| s(\xxi_j) - f_j \|_2^2 + \lambda \| s \|_{\cH}^2.
 \end{align}
 Then the \emph{regularized least-squares problem} is given as 
 
 \begin{center}
 Find $ s_{\lambda} \in \cH $ such that $ s_{\lambda} = \argmin_{s \in \cH} J_{\lambda}(s) $.
 \end{center}
\end{definition}

It turns out that, if $ \cH $ is a RKHS, the solution of the regularized least-squares problem, a minimization problem in an infinite dimensional space, is unique and an element of the finite dimensional space $ V_M $, see, e.g., \citet{Steinwart2008}.

\begin{theorem}\label{thrm:PLSRepresenterTheorem}
For any $ \lambda > 0 $, there is a unique minimizer of the regularized least-squares problem of \cref{def:PLS} in $ \cH $. 
 
In addition, there exists a coefficient vector $ \ww \in \R^{rM} $ such that 
 \begin{align*}
  s_{\lambda} = \sum_{i=1}^{M} K(\cdot, \xxi_i) \ww_i,
 \end{align*}
 i.e., $ s_{\lambda} \in V_M $.
\end{theorem}

The following estimate is important for the error analysis.

\begin{proposition}\label{prop:EstimatesPLS}
Let $s_{\lambda} \in V_M$ be the unique minimizer of the 
regularized least-squares problem and assume that the labels 
are generated by $f \in \mathcal{H}$, i.e., $f_i = f(\xxi_i)$, 
$1 \leq i \leq M$. Then
\begin{align*}
    \|s_{\lambda}\|_{\mathcal{H}} \leq \|f\|_{\mathcal{H}}.
\end{align*}
\end{proposition}

\begin{remark}[Computational form]
Both the interpolant and the regularized least-squares solution 
admit explicit representations in terms of the kernel matrix 
$\KK_{\Xi,\Xi} \in \R^{rM \times rM}$ with block 
entries $K(\xxi_i, \xxi_j) \in \mathbb{R}^{r \times r}$:
\begin{align*}
    s_0 &= \KK_{\Xi}(\cdot)\, \KK_{\Xi,\Xi}^{-1} \; [f], \\
    s_{\lambda} &= \KK_{\Xi}(\cdot) 
    \left(\KK_{\Xi,\Xi} + \lambda I\right)^{-1} [f],
\end{align*}
where $[f] = [f_1, \dots, f_M]^{\transpose} \in \mathbb{R}^{rM}$ is the vector containing the labels. 
In particular, $\mathcal{I}_\Xi$ and $\mathcal{Q}_{\lambda,\Xi}$ 
are both linear maps that can be evaluated via a single matrix-vector 
product once the respective system matrix has been factored off\-line.
\end{remark}

\begin{remark}
It is also possible to augment the functional 
$J_{\lambda}$ in \eqref{eq:PLSFunctional} by additional 
penalization terms. Of particular interest for our framework 
is a \emph{PDE residual penalty} $\mu\|\mathcal{L}s - f\|_{L_2}^2$ 
enforcing a soft PDE constraint on the reconstruction.
We revisit this extension in \cref{sec:PhysicsInformed}.
\end{remark}

Again, the solution of the regularized least-squares problem $ s_{\lambda} $ defines a linear operator $ \cQ_{\lambda,\Xi} : \R^{rM} \to V_M $ by $ \cQ_{\lambda, \Xi} ([f]) = s_{\lambda} $. If the labels are generated by a function $ f \in C(D)^r $, this can be extended to an operator $ \cQ_{\lambda,\Xi} : C(D)^r \to V_M $ by setting $ \cQ_{\lambda,\Xi}(f) = \cQ_{\lambda,\Xi}([f]) $.  

Indeed, the label generating function does not need to be in the RKHS. In the error analysis for kernel-based operator learning, we will make use of the following estimate.

\begin{proposition}\label{prop:EstimatePLSNorm}
 With the notation and assumptions of \cref{def:PLS}, let $ \cQ_{\lambda, \Xi}: \R^{rM} \to V_M $ be the regularized least-squares operator. Then the bound 
 \begin{align*}
  \| \cQ_{\lambda,\Xi}([f]) \|_{\cH} \leq \frac{1}{\sqrt{\lambda}} \| [f] \|_2
 \end{align*}
holds for all $ [f] \in \R^{rM} $.
\end{proposition}

\begin{proof}
 Since $ \cQ_{\lambda,\Xi}([f]) $ is the unique minimizer of $ J_{\lambda} $ we have
 \begin{align*}
  \lambda \| \cQ_{\lambda,\Xi}([f]) \|_{\cH}^2 \leq J_{\lambda}\left(\cQ_{\lambda,\Xi}([f]) \right) \leq J_{\lambda}(0) = \| [f] \|_2^2.
 \end{align*}
\end{proof}

In particular, \cref{prop:EstimatePLSNorm} means that 
\begin{align*}
\|\cQ_{\lambda,\Xi} \|_{\R^{rM} \to \cH} \leq \frac{1}{\sqrt{\lambda}}. 
\end{align*}
This bound is central to the error analysis of \cref{sec:ErrorAnalysis}: it shows that the reconstruction operator $\mathcal{Q}_{\lambda,\Xi}$ is bounded as a map from data space to $\mathcal{H}$, with stability constant $\frac{1}{\sqrt{\lambda}}$ that is independent of the point set $\Xi$ and the kernel $K$. This uniform bound is what allows us to control the propagation of the learning error of $A_{\mathrm{off}}$ through the reconstruction $A_{\mathrm{on}}$.

\subsection{Sampling Inequalities}

Sampling inequalities are the key tool connecting discrete 
training errors to continuous function space norms. Concretely, 
they bound a weak Sobolev norm $\|\cdot\|_{W^t_q}$ of a function 
by a combination of a strong norm $\|\cdot\|_{H^s}$ and a 
discrete norm at the training points, 
with the fill distance $h_{\Xi,D}$ controlling the relative 
weight of the two terms. In this sense, they play an analogous 
role to covering number or Rademacher complexity bounds in 
statistical learning theory \citep{Steinwart2008}: they quantify 
how well discrete information at a finite point set represents 
the continuous function.
The version we give here can be found in \citet[Theorem 2.13]{LeGia2025}. We use the notation $ (\cdot)_+ := \max(0,\cdot) $.

\begin{theorem}\label{thrm:SamplingInequality}
Let $D \subseteq \mathbb{R}^d$ be a bounded Lipschitz domain, 
$p, q \in [1,\infty]$, $s > d/2$, and $\gamma := \max(2,p,q)$. 
Then there exist constants $h_0, C > 0$ such that for all 
$\Xi \subseteq D$ with $h_{\Xi,D} < h_0$, all $f \in H^s(D)^r$, 
and all admissible $0 \leq t \leq \ell$,
\begin{align}\label{eq:SamplingInequality}
    \|f\|_{W^t_q(D)^r} \leq C\left(h_{\Xi,D}^{s-t-d \left(\frac{1}{2}-\frac{1}{q} \right)_+} 
    \|f\|_{H^s(D)^r} + h_{\Xi,D}^{\frac{d}{\gamma} - t} 
    \|f\|_{\ell_p(\Xi)^r}\right).
\end{align}
\end{theorem}

\begin{remark}
The admissible range of $t$ depends on $s$, $q$, and $d$. 
Specifically, $\ell = \ell_0 := s - d\left(\frac{1}{2}-\frac{1}{q} \right)_+$ when 
$s \in \mathbb{N}$ and either $q > 2$ with $\ell_0 \in \mathbb{N}$, 
or $q = 2$. Otherwise $\ell = \lceil \ell_0 \rceil - 1$. 
For $q = \infty$, one additionally requires $t \in \mathbb{N}_0$.
\end{remark}

We obtain the following estimates for the two learning method introduced above.

\begin{proposition}\label{prop:ErrorEstimatesInterpolationPLS}
Let $ D \subseteq \R^d $ be a bounded Lipschitz domain and let $ s > d/2 $. 
Assume that the reproducing kernel Hilbert space $ \cH $ is norm-equivalent to $ H^s(D)^r $. 
Let $ \Xi \subseteq D $ be a set of training points with sufficiently small fill distance $ h_{\Xi,D} $.

Then the following error estimates hold for all admissible $ t $ and $ q $ as in \cref{thrm:SamplingInequality} and all $  f \in H^s(D)^r$:
\begin{enumerate}
 \item\label{item:ErrorEstimatesInterpolationPLSInterpolation} For interpolation,
 \begin{align*}
  \| f - \cI_{\Xi}(f) \|_{W^t_q(D)^r} \leq C h_{\Xi,D}^{s - t - d\left(\frac{1}{2}-\frac{1}{q} \right)_+} \| f \|_{H^s(D)^r}.
 \end{align*}
\item For regularized least-squares learning,
\begin{align*}
 \| f - \cQ_{\lambda,\Xi}(f) \|_{W^t_q(D)^r} \leq C \left( h_{\Xi,D}^{s - t - d\left(\frac{1}{2}-\frac{1}{q} \right)_+} + \sqrt{\lambda} h_{\Xi,D}^{\frac{d}{\gamma} - t} \right) \| f \|_{H^s(D)^r},
\end{align*}
where $ \gamma = \max(2,q) $.
\end{enumerate}
\end{proposition}

\begin{proof}
 The interpolation case follows immediately from \cref{thrm:SamplingInequality} and \cref{thrm:PropertiesInterpolation}, using the norm equivalence of $ \cH $ and $ H^s(D)^r $. 
 
To see the statement for the penalized least-squares case, we use \eqref{eq:SamplingInequality} with $ u = f - s_{\lambda} $ and $ p = 2 $. First, by \cref{prop:EstimatesPLS} and norm equivalence of $ \cH $ and $ H^s(D)^r $, we obtain
\begin{align*}
 \| f - s_{\lambda} \|_{H^s(D)^r} 
 \leq \| f \|_{H^s(D)^r} + \| s_{\lambda,F} \|_{H^s(D)^r}
 \leq C \| f \|_{H^s(D)^r}.
\end{align*}
and second, since $ s_{\lambda} $ is the unique minimizer of $ J_{\lambda} $
\begin{align*}
\| f - s_{\lambda} \|^2_{\ell_2(\Xi)^r}
&= \sum_{j=1}^{M} \| f(\xxi_j) - s_{\lambda}(\xxi_j) \|_2^2 \\
&\leq J_{\lambda}(s_{\lambda}) \leq J_{\lambda}(f) = \lambda \| f \|^2_{\cH}
\leq C \lambda \| f \|^2_{H^s(D)^r}.
\end{align*}
\end{proof}

In the error estimates of \cref{prop:ErrorEstimatesInterpolationPLS}, 
we assumed that the target function $f$ lies in  
$\cH$, i.e., $f \in H^s(D)^r$. In practice, this assumption 
may be too restrictive: the true operator output $\cG(u)$ 
needs not to have the same smoothness as the kernel used for 
reconstruction. This situation, where the target function has 
smoothness $b$ strictly less than the kernel smoothness $s$ is known as the 
\emph{mismatch case} or \emph{escaping the native space} 
\citep{Narcowich2006}. From a machine learning perspective, this 
corresponds to a mild form of model misspecification: the 
hypothesis space $V_M$ is richer than necessary for the target 
function, yet useful approximation rates can still be recovered, 
albeit at a reduced rate reflecting the true smoothness $b$ 
rather than the kernel smoothness $s$. To give the corresponding error estimate, we define the \emph{separation radius} $ q_{\Xi} $ of $ \Xi $ to be 
\begin{align*}
 q_{\Xi} := \frac{1}{2} \min_{i \neq j} \| \xxi_i - \xxi_j \|_2,
\end{align*}
and the \emph{mesh-ratio} $ \rho_{\Xi,D} $ of $ \Xi $ in $ D $ to be 
\begin{align*}
 \rho_{\Xi,D} := \frac{h_{\Xi,D}}{q_{\Xi}}.
\end{align*}
With this, we have the following error estimates \citep[Theorem 4.5 and Theorem 4.10]{LeGia2025}.

\begin{proposition}
 Let $ D \subseteq \R^d $ be a bounded Lipschitz domain and let $ s > d/2 $. 
Assume that the reproducing kernel Hilbert space $ \cH $ is norm-equivalent to $ H^s(D)^r $. 
Let $ \Xi \subseteq D $ be a set of training points with sufficiently small fill distance $ h_{\Xi,D} $, and assume that $ \Xi $ is quasi-uniform, i.e., $ \rho_{\Xi,D} \leq c $. Let $ f \in H^{b}(D)^r $ with $ d/2 < b < s $.

Then the following error estimates hold for all admissible $ t $ and $ q $ as in \cref{thrm:SamplingInequality}:
\begin{enumerate}
\item For interpolation
\begin{align*}
 \| f - \cI_{\Xi}(f) \|_{W^t_q(D)^r} \leq C h_{\Xi,D}^{b - t - d\left(\frac{1}{2}-\frac{1}{q} \right)_+} \rho_{\Xi,D}^{s - b} \| f \|_{H^{b}(D)^r},
\end{align*}
with a constant $ C > 0 $.
\item For regularized least-squares learning
\begin{align*}
 &\hspace{-2em}\| f - \cQ_{\lambda,\Xi}(f) \|_{W^t_q(D)^r} \leq \\
 &\leq C \left(  h_{\Xi,D}^{b - t - d\left(\frac{1}{2}-\frac{1}{q} \right)_+} \rho_{\Xi,D}^{s - b} + \sqrt{\lambda} h_{\Xi,D}^{b - s + \frac{d}{\gamma} - t} \rho_{\Xi,D}^{s-b} \right) \| f \|_{H^{b}(D)^r}.
\end{align*}
\end{enumerate}
\end{proposition}

\begin{remark}
For a quasi-uniform point set $\Xi = \{ \xxi_1, \dots, \xxi_M\} $ in a 
bounded domain $D \subseteq \R^d$, i.e., 
$\rho_{\Xi,D} \leq c$, the fill distance 
satisfies $h_{\Xi,D} \sim M^{-1/d}$, where the implicit 
constants depend only on $D$ and the mesh-ratio bound $c$. 
This relation allows all fill distance conditions in this 
paper to be restated directly in terms of sample sizes.
\end{remark}

\section{Error Analysis and Budget Allocation for Kernel-based Operator Learning}\label{sec:ErrorAnalysis}

The purpose of the following analysis is to establish both convergence of the surrogate $ \cA $ and an explicit budget allocation condition quantifying how the three design parameters must be coupled. The resulting balance condition will provide an explicit resource-allocation rule for kernel-based operator learning.

We recall that 
\begin{align*}
\cA = A_{\mathrm{on}} \circ A_{\mathrm{off}} \circ S_X
\end{align*}
consists of three components:
$S_X$ discretizes the input function,
$A_{\mathrm{off}}$ learns a finite-dimensional map between 
    discretized inputs and outputs, and
$A_{\mathrm{on}}$ reconstructs an output function from 
    the predicted discretized observations.

\begin{assumption}\label{assum:ErrorAnalysis}
We make the following assumptions:
\begin{enumerate}
    \item Let $ \cU $ be a Banach space and $ \cM \subseteq \cU $ a compact model class.
    \item\label{assum:OutputSpace} The output space $\cV = 
    \cH_{K_\cV}$ is a reproducing kernel Hilbert space 
    norm-equivalent to $H^\sigma(\cD)^\ell$, $ \sigma > d/2 $.
    \item\label{assum:BottomSpace} The bottom RKHS $\cH_{K_b}$ 
    is norm-equivalent to $H^\alpha(B)^{m\ell}$, $ \alpha > np/2 $, where 
    $B = S_X(\cM) \subseteq \R^{np}$ is the image of the 
    model class under $S_X$. Let $ B $ be a bounded Lipschitz domain.
    \item\label{assum:ExistenceG} There exists a 
    function $g \in \cH_{K_b}$ satisfying 
    $g(S_X(u)) = S_Y(\cG(u))$ for all $u \in \cM$.
    \item\label{assum:OutputPoints} The output observation 
    points $Y \subseteq \cD$ have fill 
    distance $h_{Y,\cD}$.
    \item\label{assum:InputPoints} The discretized training 
    inputs $U = \{S_X(u_i)\}_{i=1}^N \subseteq B$ have fill 
    distance $h_{U,B}$.
\end{enumerate}
\end{assumption}

\cref{assum:ErrorAnalysis}(\ref{assum:ExistenceG}) formalizes the requirement 
that the operator $\cG$ can be represented consistently at 
the discretized level. In particular, it guarantees that the 
continuous operator learning problem induces a well-defined 
finite-dimensional learning problem.

The key idea underlying our error analysis, which distinguishes 
our approach from \citet{Batlle2024}, is the following: rather 
than tracing the error through the full composition 
$A_{\mathrm{on}} \circ A_{\mathrm{off}} \circ S_X$ directly, we interpret 
$\cA(u) = A_{\mathrm{on}}(A_{\mathrm{off}}(S_X(u)))$ as a \emph{reconstruction in 
$\cV$ from perturbed data}. Specifically, the exact discretized 
output $\vv^\dagger = S_Y(\cG(u)) \in \R^{m\ell}$ is unknown, 
and $A_{\mathrm{on}}$ instead receives the disturbed data 
$A_{\mathrm{off}}(S_X(u)) \approx \vv^\dagger$. This justifies the natural choice $A_{\mathrm{on}} = \cQ_{\lambda,Y}$. Since $\cQ_{\lambda,Y}$ is linear, we can insert the exact data and decompose the 
total error as
\begin{align}\label{eq:ErrorDecomposition}
    \|\cG(u) - \cA(u)\|_{W^\tau_q(\cD)^\ell} 
    \leq \underbrace{\|\cG(u) - 
    A_{\mathrm{on}}(\vv^\dagger)\|_{W^\tau_q(\cD)^\ell}}_{\text{Term } I} 
    + \underbrace{\|A_{\mathrm{on}}\|_{\R^{m\ell} \to \cV} 
    \cdot \|A_{\mathrm{off}}(S_X(u)) - 
    \vv^\dagger\|_2}_{\text{Term } II},
\end{align}
where Term $ I$ is the reconstruction error incurring when reconstructing $ \cG(u) $ from $ \vv^{\dagger}$. Its decay is governed by the output discretization density $ h_{Y,\cD} $ and  Term $ II $ is the learning error measuring the approximation error of the learned discretized operator. It is amplified by the reconstruction stability factor $\|A_{\mathrm{on}}\| \leq \lambda^{-1/2}$. This error is governed by the training discretization density $ h_{U,B} $.

The two terms are coupled only through the regularization 
parameter $\lambda$: smaller $\lambda$ reduces the 
reconstruction error in Term $I$ but amplifies the learning 
error in Term $II$ via the stability constant, and vice versa. 
This explicit coupling allows an optimal choice of $\lambda$, which we discuss after deriving general error estimates.

\subsection{Generating the Data by Interpolation}\label{subsec:GeneratingDataByInterpolation}

We first consider the noiseless setting in which the 
discretized operator is learned via kernel interpolation. 
This corresponds to deterministic training outputs generated 
without additional noise.

The following theorem quantifies how reconstruction accuracy 
and discretized learning accuracy interact through the 
regularization parameter $\lambda$.

\begin{theorem}\label{thrm:ErrorEstimateInterpolationData}
With the notation and assumptions of \cref{assum:ErrorAnalysis}, 
let $A_{\mathrm{off}} = \cI_U$ be the interpolation operator of 
\cref{subsubsec:Interpolation}. Then there exist constants 
$h_1, h_2, C_1, C_2 > 0$ such that for all $h_{Y,\cD} < h_1$, 
all $h_{U,B} < h_2$, all admissible $\tau$, $q$ as in 
\cref{thrm:SamplingInequality}, and all $u \in \cM$, the estimate
\begin{align*}
    \|\cG(u) - \cA(u)\|_{W^{\tau}_q(\cD)^{\ell}} \leq 
    C_1 \mathfrak{f}_1 \|\cG(u)\|_{H^{\sigma}(\cD)^{\ell}} + 
    C_2 \mathfrak{f}_2 \|g\|_{H^{\alpha}(B)^{m\ell}}
\end{align*}
holds, where
\begin{align*}
    \mathfrak{f}_1 := h_{Y,\cD}^{\sigma - \tau - 
    d\left(\frac{1}{2} - \frac{1}{q}\right)_+} + 
    \sqrt{\lambda}\, h_{Y,\cD}^{\frac{d}{\gamma} - \tau} 
    \quad \text{and} \quad 
    \mathfrak{f}_2 := \frac{1}{\sqrt{\lambda}}\, 
    h_{U,B}^{\alpha - \frac{1}{2}np},
\end{align*}
with $\gamma = \max(2,q)$.
\end{theorem}

\begin{proof}
The error decomposition \eqref{eq:ErrorDecomposition} 
gives Term $I$ and Term $II$. Term $I$ is bounded by 
applying \cref{prop:ErrorEstimatesInterpolationPLS} 
to $\cQ_{\lambda,Y}$ with $f = \cG(u) \in \cV$. 
Term $II$ is bounded using \cref{prop:EstimatePLSNorm} 
for $\|A_{\mathrm{on}}\|$ and \cref{prop:ErrorEstimatesInterpolationPLS} 
applied to $\cI_U$ with $f = g \in \cH_{K_b}$. 
Combining the two bounds yields the result.
\end{proof}

\begin{remark}[Mismatch cases]\label{rem:MismatchInterpolation}
\cref{thrm:ErrorEstimateInterpolationData} assumes that 
$\cG(u) \in \cV = \cH_{K_\cV}$ and $g \in \cH_{K_b}$, 
i.e., both the target output function and the discretized 
operator have smoothness matching their respective native 
spaces. In practice, either or both of these assumptions 
may fail, leading to two possible mismatch scenarios. In 
both cases, we additionally assume that the respective 
point sets are quasi-uniform, i.e., 
$\rho_{Y,\cD} := h_{Y,\cD} / q_Y \leq c$ and 
$\rho_{U,B} := h_{U,B} / q_U \leq c$ for a constant 
$c > 0$, where $q_Y$ and $q_U$ denote the separation 
radii of $Y$ and $U$, respectively.

In the \emph{output mismatch case}, $\cG(u) \in 
H^{\mu}(\cD)^\ell$ with $d/2 < \mu < \sigma$. 
The factor $\mathfrak{f}_1$ is replaced by
\begin{align*}
\frakf_1^{\mathrm{miss}} := 
h_{Y,\cD}^{\mu - \tau - d(1/2-1/q)_+} 
\rho_{Y,\cD}^{\sigma - \mu},
\end{align*}
reflecting the reduced smoothness of the target output 
at the cost of an additional mesh-ratio factor 
$\rho_{Y,\cD}^{\sigma-\mu}$. Additionally, $ \|\cG(u)\|_{H^{\sigma}(\cD)^{\ell}} $ 
has to be replaced by $ \|\cG(u)\|_{H^{\mu}(\cD)^{\ell}}$

In the \emph{bottom mismatch case}, $g \in H^{\beta}(B)^{m\ell}$ 
with $np/2 < \beta < \alpha$. The factor $\frakf_2$ 
is replaced by
\begin{align*}
\frakf_2^{\mathrm{miss}} := 
h_{U,B}^{\beta - \frac{1}{2}np} \rho_{U,B}^{\alpha - \beta},
\end{align*}
together with the respective norm of $ g $.

Both mismatches can occur simultaneously. Then 
both replacements apply.
\end{remark}

The regularization parameter $\lambda$
controls an explicit bias-variance tradeoff:
smaller values improve reconstruction accuracy, but amplify 
errors propagated from the learned discretized operator. For readability, we only consider $ q = 2 $. Similar results can be obtained for all admissible $ q $. Choosing the regularization parameter such that 
\begin{align}\label{eq:ChoiceLambdaInterpolationData}
 \lambda^* = c_{\lambda} \left(\frac{h_{Y,\cD}^{\sigma - \tau}}{h_{Y,\cD}^{\frac{d}{2} - \tau}} \right)^2 = c_{\lambda} \left( h_{Y,\cD}^{\sigma - \frac{d}{2}} \right)^2,
\end{align}
with a constant $ c_{\lambda} > 0$, this yields that 
\begin{align*}
 \frakf_1 = c_{\lambda} h_{Y,\cD}^{\sigma - \tau} \quad \text{and} \quad \frakf_2 = c_{\lambda} \frac{h_{U, B}^{\alpha - \frac{1}{2} np}}{h_{Y,\cD}^{\sigma - \frac{d}{2}}}. 
\end{align*}
This choice of $ \lambda^* $ yields an explicit scaling law relating the output discretization complexity, the training sample complexity, and the regularity of the discretized operator.

\begin{corollary}[Budget Allocation Rule]\label{cor:BudgetAllocationInterpolation}
With the notation and assumptions of 
\cref{thrm:ErrorEstimateInterpolationData}, set $q = 2$, 
assume that $Y$ and $U$ are quasi-uniform with mesh-ratios 
$\rho_{Y,\cD} \leq c$ and $\rho_{U,B} \leq c$, so that 
$h_{Y,\cD} \sim m^{-1/d}$ and $h_{U,B} \sim N^{-1/(np)}$, 
and choose $\lambda = \lambda^* $ as in 
\eqref{eq:ChoiceLambdaInterpolationData}. Then the error 
estimate becomes
\begin{align}\label{eq:ErrorNandm}
    \|\cG(u) - \cA(u)\|_{H^\tau(\cD)^\ell} \leq 
    C_1\, m^{-\frac{\sigma - \tau}{d}} \|\cG(u)\|_{H^{\sigma}(\cD)^{\ell}} + 
    C_2\, N^{-\frac{2\alpha - np}{2np}}\, 
    m^{\frac{2\sigma - d}{2d}} \|g\|_{H^{\alpha}(B)^{m\ell}},
\end{align}
and the total error converges to zero as $m, N \to \infty$ 
provided that
\begin{align}\label{eq:BudgetAllocation}
    \frac{\log N}{\log m} \geq 
    \frac{np(2\sigma - d)}{d(2\alpha - np)}.
\end{align}
\end{corollary}

The condition \eqref{eq:BudgetAllocation} expresses a 
\emph{budget allocation rule} between three free design parameters: 
 the number of training pairs $N$, the number of input observation points $n$ and the output resolution $m$.
It prescribes 
how $N$ must scale relative to $m$ and $n$ in order to prevent 
the learning error from dominating the reconstruction error. 
The interplay between these parameters is discussed further 
in \cref{subsec:ExistenceG}.

\begin{remark}[Unified Convergence rate]\label{rem:ConvergenceRate}
If $N = m^{\frac{np(2\sigma-d)}{d(2\alpha-np)}}$, i.e., 
\eqref{eq:BudgetAllocation} holds with equality, then 
both terms in \eqref{eq:ErrorNandm} decay at the same 
rate and the total error satisfies
\begin{align*}
    \|\cG(u) - \cA(u)\|_{H^\tau(\cD)^\ell} \leq 
    C\, m^{-\frac{\sigma - \tau}{d}} 
    \left(\|\cG(u)\|_{H^{\sigma}(\cD)^{\ell}} + \|g\|_{H^{\alpha}(B)^{m\ell}}\right).
\end{align*}
In other words, if $N$ grows sufficiently fast 
relative to $m$, the learning error of $A_{\mathrm{off}}$ does not 
pollute the overall convergence rate, and the surrogate 
$\cA$ converges at the same rate $m^{-(\sigma-\tau)/d}$ 
as the pure reconstruction error of $A_{\mathrm{on}}$ from exact data.
\end{remark}

\subsection{Generating the Data by Regularized Least-Squares}\label{subsec:GeneratingDataByRLS}

We now replace interpolation at the offline phase by 
regularized least-squares (RLS) approximation. This setting is 
particularly relevant in operator learning applications, 
where training outputs are often generated numerically and 
therefore contain discretization or solver error. In this 
case, exact interpolation may lead to instability or 
overfitting of the discretized operator.

The following theorem shows that we obtain a similar convergence estimate as in the 
interpolation setting, up to an additional bias term.

\begin{theorem}\label{thrm:ErrorEstimatePLSData}
With the notation and assumptions of \cref{assum:ErrorAnalysis}, 
let $A_{\mathrm{off}} = \cQ_{\mu,U}$ be the regularized least-squares 
operator of \cref{subsubsec:PLS} with regularization 
parameter $\mu > 0$. Then there exist constants 
$h_1, h_2, C_1, C_2 > 0$ such that for all $h_{Y,\cD} < h_1$, 
all $h_{U,B} < h_2$, all admissible $\tau$, $q$ as in 
\cref{thrm:SamplingInequality}, and all $u \in \cM$, the estimate
\begin{align*}
    \|\cG(u) - \cA(u)\|_{W^{\tau}_q(\cD)^{\ell}} \leq 
    C_1 \mathfrak{f}_1 \|\cG(u)\|_{H^{\sigma}(\cD)^{\ell}} + 
    C_2 \widetilde{\mathfrak{f}}_2 \|g\|_{H^{\alpha}(B)^{m\ell}}
\end{align*}
holds, where $ \mathfrak{f}_1 $ is as in \cref{thrm:ErrorEstimateInterpolationData} and $ \widetilde{\mathfrak{f}}_2 $ is given by
\begin{align*}
    \widetilde{\mathfrak{f}}_2 := \frac{1}{\sqrt{\lambda}} 
    \left(h_{U,B}^{\alpha - \frac{1}{2}np} + 
    \sqrt{\mu}\right).
\end{align*}
\end{theorem}

\begin{remark}[Mismatch cases]
The mismatch scenarios discussed in 
\cref{rem:MismatchInterpolation} extend directly to 
the this setting and yield analogous modified rates.
\end{remark}

The regularization parameter $\mu$ introduces a second bias-variance 
tradeoff at the bottom learning stage. Choosing $\mu$ too 
small reduces regularization bias but amplifies instability, 
whereas large $\mu$ improves stability at the cost of 
approximation accuracy. Choosing $\lambda$ as in 
\eqref{eq:ChoiceLambdaInterpolationData} and $\mu$ as
\begin{align}\label{eq:ChoiceMuPLSData}
    \mu^* := c_\mu \left(h_{U,B}^{\alpha - \frac{1}{2}np} 
    \right)^2,
\end{align}
with a free constant $c_\mu > 0$, exactly balances the 
fill-distance term $h_{U,B}^{\alpha - \frac{1}{2}np}$ 
and the regularization bias $\sqrt{\mu}$ in 
$\widetilde{\mathfrak{f}}_2$, yielding the following result. Remarkably, after optimal parameter balancing, the RLS and 
interpolation regimes yield identical asymptotic scaling 
laws.

\begin{corollary}\label{cor:BudgetAllocationPLS}
With the notation and assumptions of 
\cref{thrm:ErrorEstimatePLSData}, set $q = 2$, assume 
that $Y$ and $U$ are quasi-uniform with mesh-ratios 
$\rho_{Y,\cD} \leq c$ and $\rho_{U,B} \leq c$, and 
choose $\lambda = \lambda^*$ as in \eqref{eq:ChoiceLambdaInterpolationData} 
and $\mu = \mu^*$ as in \eqref{eq:ChoiceMuPLSData}. Then the 
error estimate and budget allocation condition 
\eqref{eq:ErrorNandm} and \eqref{eq:BudgetAllocation} of 
\cref{cor:BudgetAllocationInterpolation} hold verbatim.
\end{corollary}

\cref{cor:BudgetAllocationPLS} shows that, with the optimal 
choice of $\mu$, regularized least-squares at the bottom 
achieves the same asymptotic convergence rates and budget 
allocation condition as interpolation. Although RLS does not improve the asymptotic rate, it 
improves robustness in practice by better conditioning of the kernel system, particularly when training outputs contain numerical noise or the training set is large.

\begin{remark}[Comparison with \citet{Batlle2024}]
\label{rem:ComparisonBDHO}
The error results presented above are complementary 
to the main quantitative result of \citet{Batlle2024} 
(Theorem 3.3), reflecting the different goals of 
the two frameworks. We highlight three structural 
differences.

\emph{(i) Target norm.} 
We extend the range of norms the error is measured from $ H^{t'} $-norm results to more general $ W^{\tau}_q $-norm estimates. This includes in particular $ L_1 $- and $ L_{\infty} $-error results. 

\emph{(ii) Input discretization term.} We can drop the assumption that the input space $ \cU $ has RKHS structure and do not need a reconstruction $ A_{le} $. Moreover, Condition 3.2 of \citet{Batlle2024}, namely 
$u_i = A_{le} \circ S_X(u_i)$ for all training inputs is dropped. The input 
discretization error is therefore entirely absent in 
\cref{thrm:ErrorEstimateInterpolationData}.

\emph{(iii) Explicit vs implicit bottom error rates.} 
In \citet[Equation 3.8]{Batlle2024}, the learning error is 
expressed in terms of $\max_j \|f^\dagger_j\|_S$, the 
RKHS norm of the components of $f^\dagger$, which 
characterizes the smoothness of the target but does not 
directly give a rate in terms of the number of training 
pairs $N$. In \cref{thrm:ErrorEstimateInterpolationData}, 
the corresponding factor is $\mathfrak{f}_2 = 
\frac{1}{\sqrt{\lambda}} h_{U,B}^{\alpha - \frac{1}{2}np}$, 
which provides an explicit convergence rate in the fill 
distance $h_{U,B}$ of the training inputs and directly 
yields the budget allocation condition 
\eqref{eq:BudgetAllocation} of 
\cref{cor:BudgetAllocationInterpolation}.
\end{remark}

\subsection{Existence of the Discretized Operator}
\label{subsec:ExistenceG}

We expand the discussion of the existence of the discretized operator $g$ begun in \cref{subsec:FillDistanceAndExistenceOfG}, focusing here on the competing demands that existence and learnability of $g$ place on the discretization parameters $n$ and $m$.

A practically important special case is 
$\cM = B_R[0] \cap \cU_{n_0}$, where 
$\cU_{n_0} \subseteq \cU$ is a subspace of 
dimension $n_0$ and $B_R[0]$ denotes the closed ball 
of radius $R$ in $\cU$ about $ 0 $. Then $\cM$ is compact 
and $S_X: \cU_{n_0} \to \R^{np}$ is a linear 
map between finite-dimensional spaces. Whenever 
the point evaluations $\{\delta_{x_1}, \dots, 
\delta_{x_n}\}$ separate the elements of 
$\cU_{n_0}$, satisfied in particular 
when $np \geq n_0$, the map $S_X$ is injective 
on $\cU_{n_0}$, so $\iota_{\Phi} \equiv 0$ on 
$\cM$ and $g$ is well-defined. This setting 
closely resembles the framework of 
\citet{Batlle2024}, where training inputs satisfy 
the reconstruction condition $u_i = A_{le} \circ 
\phi(u_i)$, effectively restricting to a 
finite-dimensional subspace of the input RKHS. 
In practice, PCA preprocessing of the training 
inputs, as used in \cref{sec:Numerics},
realizes precisely this setting: $\cU_{n_0}$ is 
the span of the leading $n_0$ PCA components, 
and injectivity of $S_X$ on $\cU_{n_0}$ holds 
whenever the PCA basis functions are separated 
by the sampling points $X$.

\begin{remark}[Verification of \eqref{eq:AssumInequalityExistenceG}]
In this finite-dimensional setting, condition \eqref{eq:AssumInequalityExistenceG} is automatically satisfied. Since $ S_X: \cU_{n_0} \to \R^{np} $ is a linear injective map between finite-dimensional spaces, its left inverse $ S_X^{-1}: S_X(\cU_{n_0}) \to \cU_{n_0} $ is bounded. Hence, \eqref{eq:AssumInequalityExistenceG} holds with $ \iota_{\Phi} \equiv 0 $ and $ C = \| S_X^{-1} \| $.

In the PCA setting, $ C $ depends on the condition number of the matrix formed by the leading $ n_0 $ PCA basis functions at the sampling points $ X $, and is controlled by choosing $ X $ with sufficiently small fill distance relative to the spatial frequency of the basis functions.
\end{remark}

Furthermore, the existence and learnability of $g$ impose competing 
demands on $n$ and $m$. Increasing $n$ reduces 
information loss under discretization and improves 
identifiability on $\cM$: in the degenerate case 
$n = 0$, all inputs become indistinguishable. At the same time, the 
budget allocation condition \eqref{eq:BudgetAllocation} 
shows that the required number of training pairs 
grows as
\begin{align*}
N \gtrsim 
m^{\frac{np(2\sigma-d)}{d(2\alpha-np)}},
\end{align*}
so larger $n$ simultaneously increases the 
statistical complexity of learning $g$. The output 
parameter $m$ plays a different role: larger $m$ 
improves reconstruction accuracy of $A_{\mathrm{on}}$ through 
the decay of $h_{Y,\cD}$, but increases the output 
dimension $m\ell$ of $g$ and hence the complexity 
of the regression problem. Thus, $n$ and $m$ 
together control a fundamental 
representation-learnability tradeoff in the 
two-stage operator learning framework.

\section{Physics-Informed Reconstruction}
\label{sec:PhysicsInformed}

In many operator learning applications, $\cG$ maps 
a parameter or coefficient function $u$ to the 
solution $v$ of a PDE. Concretely, we 
assume that for each $u \in \cM$, the output 
$\cG(u) \in \cV$ is the unique solution of
\begin{align}\label{eq:ParametricPDE}
    \cL_u v = f(u) \quad \text{in } \cD,
\end{align}
where $\cL_u: \cV \to L_2(\cD)^\ell$ is a linear 
differential operator of order $\nu$ depending on 
$u$, and $f: \cM \to L_2(\cD)^\ell$ is a known 
forcing term. We emphasize that 
\eqref{eq:ParametricPDE} covers both the 
coefficient-to-solution map (e.g., 
$\cL_u = -\mathrm{div}(e^u \nabla\cdot)$, 
$f(u) = 1$) and the forcing-to-solution map 
(e.g., $\cL_u = -\Delta$, $f(u) = u$).

We wish the surrogate $\cA^{\mathrm{PDE}}$ to 
approximately satisfy \eqref{eq:ParametricPDE}, 
i.e., we want $\cL_u \cA^{\mathrm{PDE}}(u) 
\approx f(u)$ in $\cD$ for each $u \in \cM$. 
As recalled in \cref{sec:Introduction}, one 
classical approach encodes physical constraints 
directly into the kernel $K_\cV$. However, this 
requires the constraint to be identifiable from 
the analytical structure of $\cL_u$ and fails 
when $\cL_u$ depends on $u$. We therefore adopt 
a soft-constraint approach: we penalize the PDE 
residual at a finite set of \emph{collocation 
points} $Z = \{\zz_1, \dots, \zz_{m_\cL}\} 
\subseteq \cD$, requiring only the ability to 
evaluate $\cL_u$ and $f(u)$ pointwise for each 
$u \in \cM$.

\begin{assumption}\label{assum:PhysicsInformed}
We make the following assumptions:
\begin{enumerate}
    \item\label{assum:PI:order} The operator 
    $\cL_u$ is a linear differential operator of 
    order $\nu$ for each fixed $u \in \cM$, and 
    the native space $\cV = \cH_{K_\cV} = H^{\sigma}(\cD)^{\ell}$ satisfies 
    $\cV \hookrightarrow C^{2\nu}(\cD)$, which 
    holds whenever $\sigma > 2\nu + d/2$.
    \item\label{assum:PI:elliptic} $\cL_u$ satisfies a uniform a priori estimate over $\cM$, i.e., there 
    exists a constant $c_{\min} > 0$ such that 
    for all $u \in \cM$ and all $v \in \cV$,
    \begin{align*}
        \|v\|_{H^\nu(\cD)^\ell} \leq \frac{1}
        {c_{\min}}\left(
        \|\cL_u v\|_{L_2(\cD)^\ell} + 
        \|v\|_{L_2(\cD)^\ell}\right).
    \end{align*}
    \item\label{assum:PI:collocation} The PDE 
    collocation points $Z = \{\zz_1, \dots, 
    \zz_{m_\cL}\} \subseteq \cD$ have fill 
    distance $h_{Z,\cD}$.
\end{enumerate}
\end{assumption}

The linearity of $\cL_u$ is essential: it ensures 
that $v \mapsto \cL_u v(\zz_k)$ is a bounded 
linear functional on $\cV$, which is required 
for the representer theorem of 
\citet{Micchelli2005}. Particular examples of 
operators satisfying \cref{assum:PhysicsInformed} (2) are Darcy-type 
operators $\cL_u = -\mathrm{div}(a(u,\cdot)
\nabla\cdot)$ whenever $a(u,\xx) \geq a_{\min} 
> 0$ uniformly over $\cM$ and $\cD$, which holds 
for example when $u$ is bounded on $\cM$ and 
$a = e^u$. For second-order uniformly elliptic 
operators with suitable boundary conditions, the 
estimate follows from classical elliptic 
regularity theory \citep{Evans2010}.

\subsection{The Physics-Informed Reconstruction 
Operator and Representer Theorem}
\label{subsec:PIRepresenter}

We define the \emph{collocation sampling operator} 
$S_Z^{\cL_u}: \cV \to \R^{m_\cL \ell}$ by
\begin{align*}
    [S_Z^{\cL_u} v]_k := \cL_u v(\zz_k), 
    \quad k = 1, \dots, m_\cL,
\end{align*}
which is the natural analogue of $S_Y$ for 
    differential operator evaluations. For labels
$\vv \in \R^{m\ell}$ and $u \in \cM$, define 
the \emph{physics-informed Tikhonov functional} 
$J^{\mathrm{PDE}}_{\lambda,\mu,u}: \cV \to \R$ by
\begin{align}\label{eq:PhysicsInformedFunctional}
    J^{\mathrm{PDE}}_{\lambda,\mu,u}(s) := 
    \|S_Y s - \vv \|_2^2 + \lambda\|s\|_\cV^2 + 
    \frac{\mu}{m_\cL}
    \|S_Z^{\cL_u} s - [f(u)]\|_2^2,
\end{align}
where $[f(u)]:= S_Z(f(u)) \in \R^{m_\cL \ell}$ 
collects the values of the forcing term at the 
collocation points. The three terms in 
\eqref{eq:PhysicsInformedFunctional} penalize 
the data misfit at $Y$, the RKHS norm of $v$, 
and the PDE residual at $Z$, respectively.

\begin{definition}\label{def:PhysicsInformedRLS}
Let \cref{assum:ErrorAnalysis} and 
\cref{assum:PhysicsInformed} hold and let 
$\lambda, \mu > 0$. The \emph{physics-informed 
reconstruction operator} 
$A_{\mathrm{on}}^{\mathrm{PDE}}: \R^{m\ell} \times \cM 
\to \cV$ is defined by
\begin{align}\label{eq:PhysicsInformedOperator}
    A_{\mathrm{on}}^{\mathrm{PDE}}(\vv, u) := 
    \argmin_{s \in \cV}\, 
    J^{\mathrm{PDE}}_{\lambda,\mu,u}(s).
\end{align}
The corresponding \emph{physics-informed 
surrogate operator} $ \cA^{\mathrm{PDE}}: \cU \to \cV$ is
\begin{align*}
    \cA^{\mathrm{PDE}}(u) := 
    A_{\mathrm{on}}^{\mathrm{PDE}}(A_{\mathrm{off}}(S_X(u)), u), 
    \quad u \in \cM.
\end{align*}
\end{definition}

\begin{remark}
Note that $\cA^{\mathrm{PDE}}(u)$ is no longer 
of the form $A_{\mathrm{on}} \circ A_{\mathrm{off}} \circ S_X$ with a 
fixed $A_{\mathrm{on}}$, since $A_{\mathrm{on}}^{\mathrm{PDE}}(\cdot, u)$ 
depends on $u$ through $\cL_u$ and $f(u)$. 
However, for each fixed $u$ it is still a linear 
map from $\R^{m\ell}$ to $\cV$, and the error 
decomposition \eqref{eq:ErrorDecomposition} 
applies verbatim with $A_{\mathrm{on}}^{\mathrm{PDE}}(\cdot, 
u)$ in place of $A_{\mathrm{on}}$.
\end{remark}

Since both $S_Y$ and $S_Z^{\cL_u}$ consist of 
bounded linear functionals on $\cV$, the 
generalized representer theorem of 
\citet{Micchelli2005} applies to 
\eqref{eq:PhysicsInformedOperator}.

\begin{theorem}\label{thrm:PIRepresenter}
Let \cref{assum:ErrorAnalysis} and 
\cref{assum:PhysicsInformed} hold. Then for each 
$\vv \in \R^{m\ell}$ and $u \in \cM$ there 
exists a unique minimizer $A_{\mathrm{on}}^{\mathrm{PDE}}
(\vv, u)$ of \eqref{eq:PhysicsInformedOperator}, 
which lies in the \emph{collocation-augmented 
hypothesis space}
\begin{align*}
    V_{Y,Z}(u) := \spn\left(
    \left\{K_\cV(\cdot, \yy_j)\right\}_{j=1}^m 
    \cup 
    \left\{(\cL_{u}^{(2)} K_\cV)
    (\cdot, \zzeta)\big|_{\zzeta=\zz_k}
    \right\}_{k=1}^{m_\cL} \right),
\end{align*}
where $\cL_{u}^{(2)}$ denotes $\cL_u$ acting 
on the second argument of $K_\cV$. Writing
\begin{align*}
    A_{\mathrm{on}}^{\mathrm{PDE}}(\vv, u) = 
    \sum_{j=1}^m  
    K_\cV(\cdot, \yy_j) \aalpha_j + 
    \sum_{k=1}^{m_\cL} (
    (\cL_{u}^{(2)} K_\cV)(\cdot, \zzeta)
    \big|_{\zzeta=\zz_k})\bbeta_k ,
\end{align*}
the coefficients $(\aalpha, 
\bbeta) \in \R^{m\ell} \times 
\R^{m_\cL\ell}$ are the unique solution of the 
block linear system
\begin{align}\label{eq:PIBlockSystem}
    \begin{pmatrix}
        \KK_{YY} + \lambda \II & 
        \KK_{YZ}^{\cL} \\[4pt]
        (\KK_{YZ}^{\cL})^\transpose & 
        \KK_{ZZ}^{\cL\cL} + 
        \dfrac{\lambda m_\cL}{\mu}\II
    \end{pmatrix}
    \begin{pmatrix} \boldsymbol{\alpha} \\[4pt]
    \boldsymbol{\beta} \end{pmatrix}
    = \begin{pmatrix} \vv \\[4pt]
    \ff(u) \end{pmatrix},
\end{align}
where the kernel matrices are defined entry-wise 
by
\begin{align*}
    [\KK_{YY}]_{ij} &:= 
    K_\cV(\yy_i, \yy_j), \\
    [\KK_{YZ}^{\cL}]_{jk} &:= 
    (\cL_{u}^{(2)} K_\cV)(\yy_j, \zzeta)
    \big|_{\zzeta=\zz_k}, \\
    [\KK_{ZZ}^{\cL\cL}]_{kl} &:= 
    (\cL_{u}^{(1)} \cL_{u}^{(2)} K_\cV)
    (\xxi,\zzeta)
    \big|_{\zzeta=\zz_k,\,\xxi=\zz_l}.
\end{align*}
\end{theorem}

\begin{proof}
Let $ u \in \cM $ be fixed. Because of the strict convexity of $ \| \cdot \|_{\cV} $ there is a unique minimizer $ s^* $ of $ J^{\mathrm{PDE}}_{\lambda,\mu,u} $ in $ \cV $. 

To show that $ s^* \in V_{Y,Z}(u) $, we split $ s^* = s^{\parallel} + s^{\perp} $, where $ s^{\parallel} \in V_{Y,Z} $ and $ s^{\perp} \in V_{Y,Z}(u)^{\perp} $, the orthogonal complement of $ V_{Y,Z} $ in $ \cV $. Because $ V_{Y,Z}(u) $ is finite dimensional, this splitting is direct. We have, with the reproducing property of $ K_{\cV} $, that
\begin{align*}
s^{\perp}(\yy_i) = \langle s^{\perp}, K_{\cV}(\cdot, \yy_i) \rangle_{\cV} = 0, \quad \yy_i \in Y
\end{align*}
and
\begin{align*}
\| s^* \|_{\cV}^2 = \| s^{\parallel} \|_{\cV}^2 + \| s^{\perp} \|_{\cV}^2 + 2 \langle s^{\parallel}, s^{\perp} \rangle_{\cV} = \| s^{\parallel} \|_{\cV}^2 + \| s^{\perp} \|_{\cV}^2.
\end{align*}
Furthermore, by \cref{assum:PhysicsInformed}, the mapping $ h \mapsto (\cL_u h)(\zz_k) $, $ h \in \cV $, is bounded and linear. This means that there is a $ r_k \in \cV $ such that 
\begin{align*}
(\cL_u h)(\zz_k) = \langle h, r_k \rangle_{\cV}
\end{align*}
by the Riesz representation theorem. Using the reproducing property of $ K_{\cV} $, we obtain that $ r_k = \cL_u^{(2)}K_{\cV}(\cdot, \zzeta)|_{\zzeta = \zz_k} \in V_{Y,Z}(u)$. This also yields that $ \langle s^{\perp}, r_k \rangle = (\cL_u s^{\perp})(\zz_k) = 0$ for all $ \zz_k \in Z $. 

Together, we see that 
\begin{align*}
J^{\mathrm{PDE}}_{\lambda,\mu,u}(s^*) &= \| S_Y(s^*) - \vv \|_2^2 + \lambda \| s^* \|_{\cV}^2 + \frac{\mu}{m_{\cL}} \| S_Z^{\cL_u}(s^*) - [f(u)] \|_2^2 \\
&= \sum_{j=1}^{m} \| \vv_j - s^{\parallel}(\yy_j) \|_2^2 + \lambda \| s^{\parallel}\|_{\cV}^2 + \lambda \| s^{\perp} \|_{\cV}^2 + \frac{\mu}{m_{\cL}} \| S_Z^{\cL_u} ( s^{\parallel}) - [f(u)] \|_2^2.
\end{align*}
However, since $ s^* $ is the unique minimizer of $ J^{\mathrm{PDE}}_{\lambda,\mu,u} $, this means that $ \| s^{\perp} \|_{\cV}^2 = 0 $, i.e., $ s^{\perp} = 0 $. In turn, this yields $ s^* \in V_{Y,Z}(u) $. Hence, there are coefficient vectors $ \aalpha \in \R^{m\ell} $ and $ \bbeta \in \R^{m_{\cL}\ell} $ such that 
\begin{align*}
s^* = \sum_{j=1}^{m} K_{\cV}(\cdot, \yy_j) \aalpha_j + \sum_{k=1}^{m_{\cL}} (\cL_u^{(2)} K_{\cV}(\cdot, \zzeta)|_{\zzeta = \zz_k}) \bbeta_k = (\KK_Y(\cdot), \KK_{Z}^{\cL}(\cdot)) \begin{pmatrix} \aalpha \\ \bbeta \end{pmatrix}.
\end{align*}

Inserting the ansatz for $ s^* $ into $ J_{\lambda,\mu,u}^{\mathrm{PDE}} $ yields a finite-dimensional quadratic functional in $ (\aalpha, \bbeta) \in \R^{m \ell} \times \R^{m_{\cL} \ell}$. Setting the gradient to zero yields \eqref{eq:PIBlockSystem}.
\end{proof}

\begin{remark}[Smoothness requirement]
\label{rem:SmoothnessDoubling}
The kernel matrices in \eqref{eq:PIBlockSystem} 
reveal why we require 
$\sigma > 2\nu + d/2$ rather than $\sigma > \nu 
+ d/2$. The matrix $\KK_{ZZ}^{\cL\cL}$ involves 
applying $\cL_u$ to both arguments of $K_\cV$ 
simultaneously, requiring $2\nu$ derivatives of 
the kernel to exist and be bounded. This is the 
standard smoothness-doubling phenomenon in 
symmetric kernel collocation \citep{Schaback2009, 
Chen2021}. For Mat\'ern kernels, 
$\sigma > 2\nu + d/2$ is satisfied for example 
by choosing $\sigma = 2\nu + d/2 + 1/2$; for 
second-order operators ($\nu = 2$) in two 
dimensions ($d = 2$), this requires $\sigma > 5$.
\end{remark}

\begin{remark}[Connection to kernel collocation]
\label{rem:KernelCollocation}
In the limit $\mu \to \infty$, the diagonal 
regularization $\frac{\lambda m_\cL}{\mu}\II$ 
in the lower-right block of 
\eqref{eq:PIBlockSystem} tends to zero, and the 
system enforces $S_Z^{\cL_u} v = \ff(u)$ 
exactly, recovering symmetric kernel collocation 
\citep{Schaback2009} augmented with data fit at 
$Y$. In the limit $\mu \to 0$, the off-diagonal 
blocks $\KK_{YZ}^\cL$ and $(\KK_{YZ}^\cL)^\transpose$ 
become negligible relative to the diagonal 
regularization, and the system reduces to the 
standard Tikhonov system of 
\cref{subsubsec:PLS} with $\boldsymbol{\beta} 
= 0$. The parameter $\mu$ therefore continuously 
interpolates between pure data fitting 
($\mu = 0$) and exact PDE enforcement 
($\mu \to \infty$).
\end{remark}

\subsection{Computational Cost}
\label{subsec:PIDiscussion}

We now briefly discuss the computational cost of the physics informed operator learning method. The costs for the non-physics informed method can be recovered by setting $ m_{\cL} = 0 $.

The dominant online cost of $A_{\mathrm{on}}^{\mathrm{PDE}}$ 
is assembling and solving the block system 
\eqref{eq:PIBlockSystem} of size 
$(m + m_\cL) \times (m + m_\cL)$. The matrix 
$\KK_{YY}$ is the standard output kernel matrix 
and is independent of $u$. It is assembled and 
factored once per output grid $Y$. The matrices 
$\KK_{YZ}^\cL$ and $\KK_{ZZ}^{\cL\cL}$ require 
evaluating derivatives of $K_\cV$ under $\cL_u$: 
for a second-order operator this means second 
derivatives of $K_\cV$, available analytically 
for Mat\'ern kernels or via automatic 
differentiation. The assembly of 
$\KK_{ZZ}^{\cL\cL}$ costs $O(m_\cL^2)$ kernel 
evaluations and is the dominant assembly 
cost when $\cL_u$ depends on $u$. The Cholesky 
factorization of the full block system costs 
$O((m + m_\cL)^3)$, the same order as the 
standard Tikhonov solve but with a larger 
constant.

When $\cL_u = \cL$ is independent of $u$ 
(e.g., the forcing-to-solution map 
$\cL_u = -\Delta$, $f(u) = u$), all three 
kernel matrices are independent of $u$ and can 
be assembled and factored once offline. The 
online cost per test input then reduces to a 
single triangular solve of size $(m + m_\cL)$ 
plus one evaluation of $f(u)$ at $Z$. When 
$\cL_u$ depends on $u$ (e.g., the 
coefficient-to-solution map 
$\cL_u = -\mathrm{div}(e^u \nabla\cdot)$), 
the matrices $\KK_{YZ}^\cL$ and 
$\KK_{ZZ}^{\cL\cL}$ must be recomputed for 
each test input $u$, and the full 
$O((m + m_\cL)^3)$ solve is required online.

A rigorous error analysis of 
$\cA^{\mathrm{PDE}}$, including the optimal 
choice of $\mu$ and the resulting convergence 
rate, is left for future work.

\section{Numerical Experiments}\label{sec:Numerics}
We validate \cref{sec:ErrorAnalysis} on the Darcy flow equation and
test the physics-informed reconstruction of \cref{sec:PhysicsInformed}
on the Poisson equation.

\paragraph{Shared setup.}
Inputs $u$ are drawn from a Gaussian random field obtained by applying a
Gaussian filter with bandwidth $\sigma_{\mathrm{input}}=8$ to white noise and
clipping to $[-4,4]$, giving $C^\infty$-like realizations.
Solutions are computed by a finite difference scheme on a uniform
$64\times 64$ interior grid with zero Dirichlet boundary conditions.
We generate $N_{\mathrm{total}}=60000$ input-output pairs, split into
$N_{\mathrm{train}}=50000$ training and $N_{\mathrm{test}}=10000$ test pairs.
Input functions are reduced via PCA, retaining $99\%$ of variance. This is achieved with $n_{\mathrm{PCA}}=28$ components.
Output observation points $Y$ form a uniform $\sqrt{m}\times\sqrt{m}$
grid of interior points on $(0,1)^2$.

For the PI experiments, error averages are computed over a subset of
$500$ test samples due to the additional per-sample cost of solving
the augmented system.

\subsection{Kernel Operator Learning: Darcy Flow}\label{sec:NumDarcy}

We consider the parametric Darcy flow problem
\begin{align*}
  -\mathrm{div}(e^{u(x)}\nabla v(x)) = 1
  \quad\text{on }(0,1)^2,\qquad
  v = 0\text{ on }\partial(0,1)^2,
\end{align*}
where $u$ is the log-permeability and $\cG\colon u\mapsto v$ the
parameter-to-solution map.
Both the output kernel $K_\cV$ and the bottom kernel $K_b$ are matrix valued kernels with the Matérn-$\nicefrac{3}{2}$ function on the diagonal.
The native space of Matérn-$\nicefrac{3}{2}$ on $\R^2$ is $H^2(\R^2)$,
giving $\sigma=2$.
The regularization parameter is $\lambda(m)=m^{-5/2}$, which lies
well within the flat plateau of the error curve identified in
\cref{fig:LambdaSensitivity}.

\paragraph{Convergence in $m$.}

\begin{figure}[ht]
\centering
\begin{tikzpicture}
\begin{loglogaxis}[
    width=0.85\textwidth,
    height=0.45\textwidth,
    xlabel={Output observation points $m$},
    ylabel={Mean relative $L_2$ error},
    xmin=7, xmax=1100,
    ymin=3e-3, ymax=0.8,
    xtick={9,25,49,121,225,441,784},
    xticklabels={$9$,$25$,$49$,$121$,$225$,$441$,$784$},
    legend style={
        at={(0.5,-0.28)}, anchor=north,
        legend columns=-1, font=\scriptsize,
        column sep=8pt,
    },
    grid=major,
    grid style={gray!20},
    tick align=outside,
]
\addplot[blue, thick, mark=*, mark size=2pt]
coordinates {
    (9,   3.524e-01)
    (25,  1.411e-01)
    (49,  8.203e-02)
    (121, 3.678e-02)
    (225, 2.057e-02)
    (441, 1.039e-02)
    (784, 5.393e-03)
};
\addlegendentry{Oracle $A_{\mathrm{on}}$}

\addplot[orange!80!black, thick, dashed, mark=square*, mark size=2pt]
coordinates {
    (9,   3.524e-01)
    (25,  1.411e-01)
    (49,  8.204e-02)
    (121, 3.679e-02)
    (225, 2.059e-02)
    (441, 1.043e-02)
    (784, 5.462e-03)
};
\addlegendentry{Full surrogate $\mathcal{A}$}

\addplot[black, thin, dotted, no marks]
coordinates {
    (9,   4.503e-01)
    (25,  1.621e-01)
    (49,  8.270e-02)
    (121, 3.349e-02)
    (225, 1.801e-02)
    (441, 9.189e-03)
    (784, 5.169e-03)
};
\addlegendentry{Ref.\ slope $-1$}

\end{loglogaxis}
\end{tikzpicture}
\caption{Convergence of the oracle and full surrogate error as a function of
output resolution~$m$ with $N = 50{,}000$ fixed.
The oracle error transitions from a pre-asymptotic regime (local slope
$\approx -0.87$ for $m\leq 121$) toward the theoretical rate $m^{-1}$
for Matérn-$\tfrac{3}{2}$ ($\sigma=2$, $d=2$), with local slope
$-1.03$ for $m\geq 121$.
The full surrogate tracks the oracle throughout.}
\label{fig:Convergence}
\end{figure}
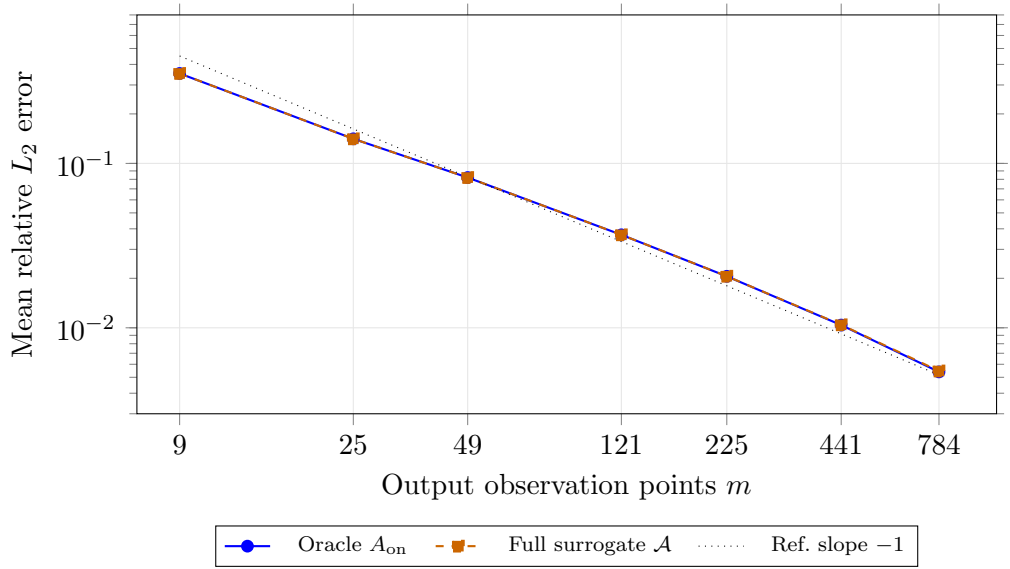
\cref{fig:Convergence} shows the mean relative $L_2$ error as a
function of $m$ with $N=N_{\mathrm{train}}$ fixed.
The oracle error exhibits a clear pre-asymptotic regime for small $m$,
with local slopes around $-0.87$ for $m\leq 121$, steepening toward
the theoretical prediction $m^{-\sigma/d}=m^{-1}$ of
\cref{thrm:ErrorEstimateInterpolationData} ($\sigma=2$, $d=2$) as $m$
increases. The local slope over $m\in\{121,\ldots,784\}$ is $-1.03$,
consistent with the asymptotic rate.
The full surrogate error tracks the oracle throughout the entire range,
confirming that the bottom-level learning error is negligible when $N$
is large.

\paragraph{Budget allocation.}
\cref{fig:BudgetAllocation} shows the full surrogate
error for $N = m^\kappa$ with varying $\kappa$, using
$n_\mathrm{PCA} = 20$.
For $\kappa = 0.5$ the error stagnates, indicating
that the training set is too small to keep pace with
the growing output resolution.
For $\kappa \geq 1.0$ the curves converge at a rate
comparable to the oracle, and the three curves with
$\kappa \in \{1.0, 1.5, 2.0\}$ are essentially
indistinguishable from one another.
This suggests that a moderate superlinear growth
$N \sim m^\kappa$ with $\kappa$ somewhere between
$0.5$ and $1.0$ is sufficient to avoid stagnation,
and that investing further in training data beyond
that threshold yields no measurable benefit.

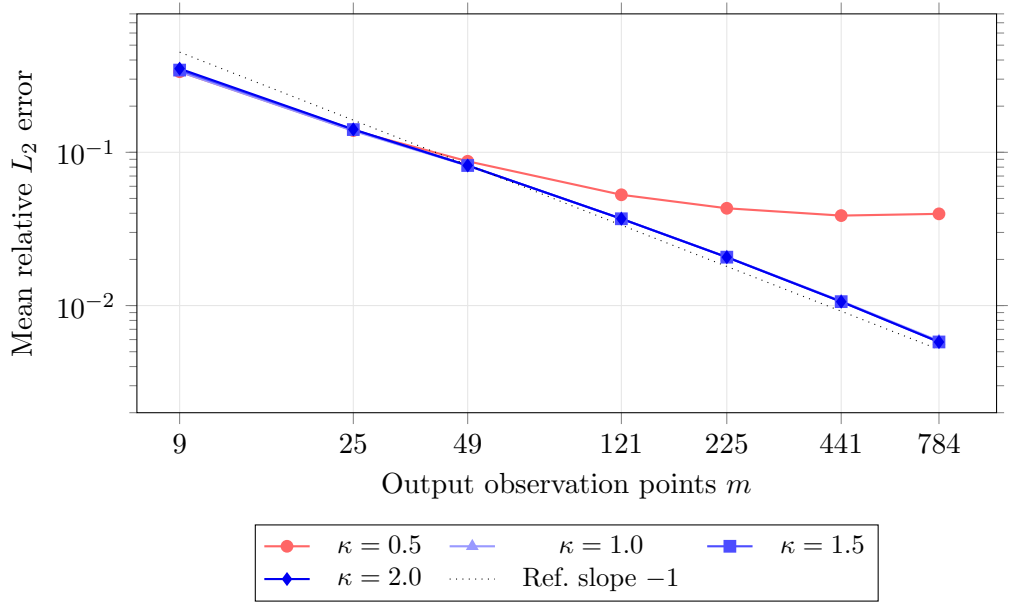
\begin{figure}[ht]
\centering
\begin{tikzpicture}
\begin{loglogaxis}[
    width=0.85\textwidth,
    height=0.45\textwidth,
    xlabel={Output observation points $m$},
    ylabel={Mean relative $L_2$ error},
    xmin=7, xmax=1100,
    ymin=2e-3, ymax=0.8,
    xtick={9,25,49,121,225,441,784},
    xticklabels={$9$,$25$,$49$,$121$,$225$,$441$,$784$},
    legend style={
        at={(0.5,-0.28)},
        anchor=north,
        font=\small,
        legend columns=3,
        column sep=8pt
    },
    grid=major,
    grid style={gray!20},
    tick align=outside,
]

\addplot[red!60, thick, mark=*, mark size=2pt]
coordinates {
    (9,   3.355e-01)
    (25,  1.386e-01)
    (49,  8.738e-02)
    (121, 5.282e-02)
    (225, 4.313e-02)
    (441, 3.865e-02)
    (784, 3.967e-02)
};
\addlegendentry{$\kappa=0.5$}

\addplot[blue!40, thick, mark=triangle*, mark size=2pt]
coordinates {
    (9,   3.355e-01)
    (25,  1.380e-01)
    (49,  8.170e-02)
    (121, 3.691e-02)
    (225, 2.076e-02)
    (441, 1.070e-02)
    (784, 5.876e-03)
};
\addlegendentry{$\kappa=1.0$}

\addplot[blue!70, thick, mark=square*, mark size=2pt]
coordinates {
    (9,   3.438e-01)
    (25,  1.407e-01)
    (49,  8.199e-02)
    (121, 3.684e-02)
    (225, 2.069e-02)
    (441, 1.061e-02)
    (784, 5.789e-03)
};
\addlegendentry{$\kappa=1.5$}

\addplot[blue!95!black, thick, mark=diamond*, mark size=2pt]
coordinates {
    (9,   3.513e-01)
    (25,  1.411e-01)
    (49,  8.206e-02)
    (121, 3.685e-02)
    (225, 2.068e-02)
    (441, 1.060e-02)
    (784, 5.785e-03)
};
\addlegendentry{$\kappa=2.0$}

\addplot[black, thin, dotted, no marks]
coordinates {
    (9,   4.503e-01)
    (25,  1.621e-01)
    (49,  8.270e-02)
    (121, 3.349e-02)
    (225, 1.801e-02)
    (441, 9.189e-03)
    (784, 5.169e-03)
};
\addlegendentry{Ref.\ slope $-1$}

\end{loglogaxis}
\end{tikzpicture}
\caption{Budget allocation experiment with $N = m^\kappa$ and
$n_\mathrm{PCA}=20$.
The curve with $\kappa=0.5$ stagnates, while
all curves with $\kappa \geq 1.0$ converge at the oracle rate
and are essentially indistinguishable from each other.}
\label{fig:BudgetAllocation}
\end{figure}

\paragraph{Regularization sensitivity.}
\begin{figure}[ht]
\centering
\begin{tikzpicture}
\begin{semilogxaxis}[
    width=0.85\textwidth,
    height=0.45\textwidth,
    xlabel={Regularization parameter $\lambda$},
    ylabel={Mean relative $L_2$ error},
    xmin=5e-15, xmax=200,
    xtick={1e-14,1e-12,1e-10,1e-8,1e-6,1e-4,1e-2,1e0,1e2},
    ymin=0, ymax=0.65,
    legend style={
        at={(0.5,-0.28)}, anchor=north,
        legend columns=-1, font=\scriptsize,
        column sep=8pt,
    },
    grid=major,
    grid style={gray!20},
    tick align=outside,
]

\addplot[blue, thick, mark=*, mark size=1.8pt]
coordinates {
    (9.999e-15,  2.057e-02)
    (3.562e-14,  2.057e-02)
    (1.269e-13,  2.057e-02)
    (4.520e-13,  2.057e-02)
    (1.610e-12,  2.057e-02)
    (5.736e-12,  2.057e-02)
    (2.043e-11,  2.057e-02)
    (7.279e-11,  2.057e-02)
    (2.593e-10,  2.057e-02)
    (9.237e-10,  2.057e-02)
    (3.290e-09,  2.057e-02)
    (1.172e-08,  2.057e-02)
    (4.175e-08,  2.057e-02)
    (1.487e-07,  2.057e-02)
    (5.298e-07,  2.058e-02)
    (1.887e-06,  2.057e-02)
    (6.723e-06,  2.058e-02)
    (2.395e-05,  2.059e-02)
    (8.532e-05,  2.064e-02)
    (3.039e-04,  2.080e-02)
    (1.083e-03,  2.134e-02)
    (3.857e-03,  2.298e-02)
    (1.374e-02,  2.711e-02)
    (4.894e-02,  3.560e-02)
    (1.743e-01,  5.091e-02)
    (6.210e-01,  7.645e-02)
    (2.212e+00,  1.165e-01)
    (7.880e+00,  1.838e-01)
    (2.807e+01,  3.232e-01)
    (1.000e+02,  5.772e-01)
};
\addlegendentry{Oracle $A_r$}

\addplot[black, dashed, thick, no marks]
coordinates {(4.4444e-03, 0) (4.4444e-03, 0.65)};
\addlegendentry{$\lambda^* = 4.44\times10^{-3}$ (theory)}

\end{semilogxaxis}
\end{tikzpicture}
\caption{Oracle error as a function of $\lambda$ for fixed $m=225$.
The error is flat over more than ten orders of magnitude to the left
of the theoretical optimum $\lambda^* = 225^{-1} \approx 4.44\times 10^{-3}$,
confirming robustness of the regularization schedule $\lambda(m)=m^{-5/2}$.}
\label{fig:LambdaSensitivity}
\end{figure}
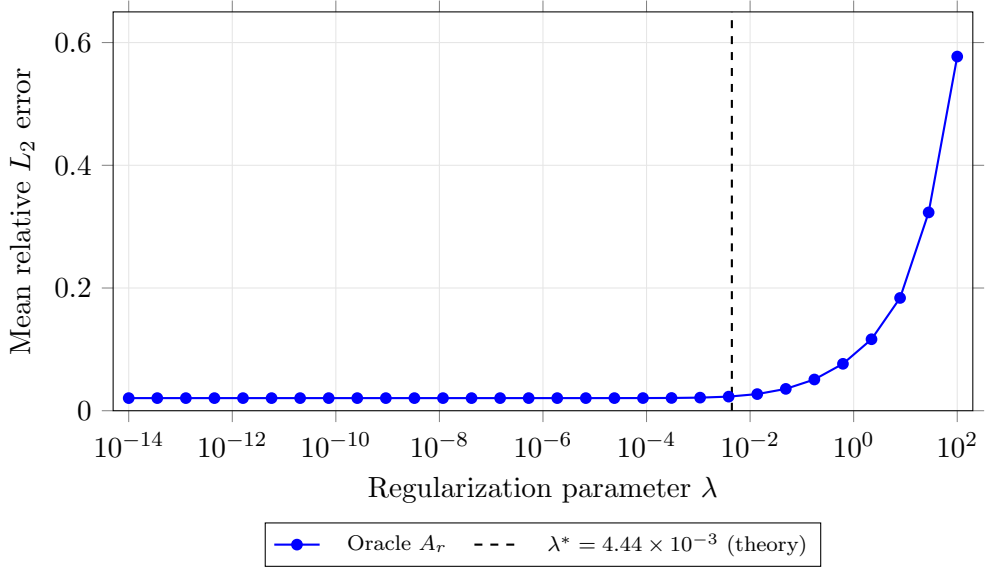

\cref{fig:LambdaSensitivity} shows the oracle error as a function of
$\lambda$ at $m=225$.
The curve is flat over many orders of magnitude around the theoretical
optimum $\lambda^*=225^{-1}\approx 4.44\times 10^{-3}$, with a
sharp increase only for $\lambda\gg\lambda^*$. The schedule used in practice, $\lambda(m)=m^{-5/2}$, gives $ \lambda(225) \approx 1.32 \times 10^{-6} $, which is well within the flat region.

\subsection{Physics-Informed Reconstruction: Poisson Equation}
\label{sec:NumPI}

We consider the Poisson problem
\begin{align*}
  -\Delta v(x) = u(x)\quad\text{on }(0,1)^2,\qquad
  v = 0\text{ on }\partial(0,1)^2.
\end{align*}
The choice of $-\Delta$ as the differential operator is deliberate:
its action on the Matérn kernel has a closed-form expression, so the
representer-theorem system of \cref{sec:PhysicsInformed} can be
assembled exactly without approximating kernel derivatives. Both matrix valued kernels used are again diagonal kernels, where the
the output kernel is chosen to be Matérn-$\nicefrac{9}{2}$ ($\nu=4.5$,
$\sigma=5.5$), because the system matrix requires evaluating
$(-\Delta)^2 K_\cV$, which demands $K_\cV\in C^4$. Matérn-$\nicefrac{9}{2}$
provides $C^8$ smoothness, giving a comfortable margin.
The lengthscale is selected per test sample by leave-one-out
cross-validation on the $Y$-block.
The PDE constraint weight is $\mu=\lambda(m)$, tying the PI
regularization to the data regularization so that neither block
dominates across the full range of $m$.
The bottom kernel is Matérn-$\nicefrac{3}{2}$. We compare three
reconstructors: the plain oracle $A_{\mathrm{on}}$ (kernel interpolant using
exact $S_Y v^*$), the PI oracle $A_{\mathrm{on}}^{\mathrm{PI}}$ (same plus PDE
enforced at $m_{\cL}$ collocation points $Z$), and the PI
surrogate (same as PI oracle but with $S_Y v^*$ replaced by $A_{\mathrm{off}}(u^*)$).
Results are averaged over $500$ test samples.

\paragraph{Convergence in $m$.}

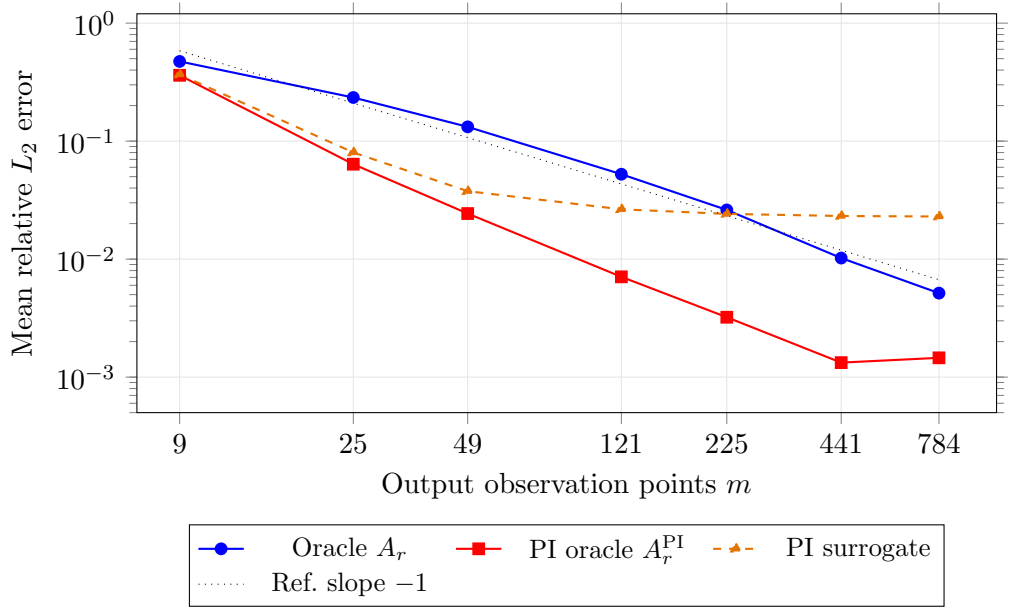
\begin{figure}[ht]
\centering
\begin{tikzpicture}
\begin{loglogaxis}[
    width=0.85\textwidth,
    height=0.45\textwidth,
    xlabel={Output observation points $m$},
    ylabel={Mean relative $L_2$ error},
    xmin=7, xmax=1100,
    ymin=5e-4, ymax=1.2,
    xtick={9,25,49,121,225,441,784},
    xticklabels={$9$,$25$,$49$,$121$,$225$,$441$,$784$},
    legend style={
        at={(0.5,-0.28)},
        anchor=north,
        font=\small,
        legend columns=3,
        column sep=8pt
    },
    grid=major,
    grid style={gray!20},
    tick align=outside,
]

\addplot[blue, thick, mark=*, mark size=2pt]
coordinates {
    (9,   4.734e-01)
    (25,  2.342e-01)
    (49,  1.321e-01)
    (121, 5.241e-02)
    (225, 2.611e-02)
    (441, 1.021e-02)
    (784, 5.149e-03)
};
\addlegendentry{Oracle $A_r$}

\addplot[red, thick, mark=square*, mark size=2pt]
coordinates {
    (9,   3.613e-01)
    (25,  6.374e-02)
    (49,  2.433e-02)
    (121, 7.076e-03)
    (225, 3.217e-03)
    (441, 1.327e-03)
    (784, 1.457e-03)
};
\addlegendentry{PI oracle $A_r^{\mathrm{PI}}$}

\addplot[orange!90!black, thick, dashed, mark=triangle*, mark size=2pt]
coordinates {
    (9,   3.668e-01)
    (25,  8.014e-02)
    (49,  3.765e-02)
    (121, 2.643e-02)
    (225, 2.417e-02)
    (441, 2.320e-02)
    (784, 2.296e-02)
};
\addlegendentry{PI surrogate}

\addplot[black, thin, dotted, no marks]
coordinates {
    (9,   5.829e-01)
    (25,  2.098e-01)
    (49,  1.071e-01)
    (121, 4.336e-02)
    (225, 2.332e-02)
    (441, 1.190e-02)
    (784, 6.691e-03)
};
\addlegendentry{Ref.\ slope $-1$}

\end{loglogaxis}
\end{tikzpicture}
\caption{Oracle, PI oracle, and PI surrogate errors
as a function of output resolution~$m$ ($m_{\cL}=m$,
$\mu=\lambda(m)$, lengthscale by LOO-CV).
The PI oracle is smaller than the plain oracle at every tested $m$;
both saturate for $m>441$ at the $O(h^2)$ FD discretization floor.
The PI surrogate saturates at $\approx 0.023$ from $m=121$,
reflecting the irreducible learning error of $A_{\mathrm{off}}$.}
\label{fig:PIConvergence}
\end{figure}

The PI oracle error is smaller than the plain oracle error at every
tested value of $m$, demonstrating that enforcing the PDE at
collocation points reduces the reconstruction error.
Both the PI oracle and plain oracle saturate for $m>441$ at an error
of approximately $1.3\times 10^{-3}$, which coincides with the
$O(h^2)$ discretization error of the $64\times 64$ FD solver; this is
a solver precision ceiling rather than a limitation of the
reconstruction method.
The PI surrogate tracks the PI oracle for small $m$ but saturates at
$\approx 0.023$ from $m=121$ onward.
This saturation reflects the irreducible learning error of $A_{\mathrm{off}}$ in
$n_{\mathrm{PCA}}=28$ input dimensions: as the PI reconstruction
becomes more accurate, it faithfully propagates the fixed prediction
bias of $A_{\mathrm{off}}$ into the output, consistent with the budget allocation
analysis of \cref{rem:ConvergenceRate}.

\paragraph{Sensitivity to $m_{\cL}$.}
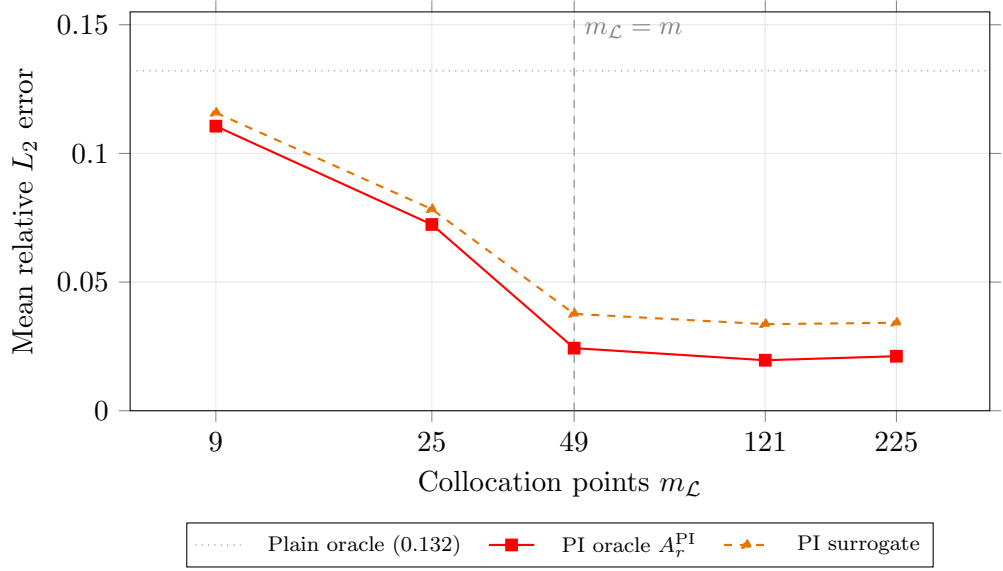
\begin{figure}[ht]
\centering
\begin{tikzpicture}
\begin{semilogxaxis}[
    width=0.85\textwidth,
    height=0.45\textwidth,
    xlabel={Collocation points $m_{\cL}$},
    ylabel={Mean relative $L_2$ error},
    xmin=6, xmax=350,
    ymin=0, ymax=0.155,
    xtick={9,25,49,121,225},
    xticklabels={$9$,$25$,$49$,$121$,$225$},
    legend style={
        at={(0.5,-0.28)}, anchor=north,
        legend columns=-1, font=\scriptsize,
        column sep=8pt,
    },
    grid=major,
    grid style={gray!20},
    tick align=outside,
]

\addplot[steelblue, thin, dotted, no marks, domain=6:350]
    {0.1321};
\addlegendentry{Plain oracle ($0.132$)}

\addplot[red, thick, mark=square*, mark size=2pt]
coordinates {
    (9,   1.106e-01)
    (25,  7.238e-02)
    (49,  2.433e-02)
    (121, 1.963e-02)
    (225, 2.119e-02)
};
\addlegendentry{PI oracle $A_r^{\mathrm{PI}}$}

\addplot[orange!90!black, thick, dashed, mark=triangle*, mark size=2pt]
coordinates {
    (9,   1.158e-01)
    (25,  7.832e-02)
    (49,  3.764e-02)
    (121, 3.365e-02)
    (225, 3.420e-02)
};
\addlegendentry{PI surrogate}

\addplot[gray, thin, dashed, no marks]
    coordinates {(49, 0) (49, 0.155)};
\node[gray, font=\small] at (axis cs:49,0.148) [anchor=west] {$m_{\cL}=m$};

\end{semilogxaxis}
\end{tikzpicture}
\caption{Effect of the number of collocation points $m_{\cL}$
at fixed output resolution $m=49$ ($\mu=\lambda(m)$, lengthscale by
LOO-CV).
The PI oracle error decreases up to $m_{\cL}=m=49$ and then
stagnates; adding further collocation points yields no benefit.
The choice $m_{\cL}=m$ is therefore both sufficient and
essentially optimal.}
\label{fig:PInCol}
\end{figure}

\cref{fig:PInCol} shows the PI oracle error at $m=49$ as a function
of the number of collocation points.
The error decreases up to $m_{\cL}=m=49$ and then stagnates.
Adding collocation points beyond the output resolution gives no
further benefit.
This confirms that $m_{\cL}=m$ is a sufficient and
essentially optimal choice.

\section{Conclusion}\label{sec:Conclusion}

We studied kernel-based operator learning in the two-stage
sampling framework of \citet{Batlle2024} and \citet{Sharma2026} and derived an explicit budget allocation rule relating the number $N$ of training pairs, the number of input observations $n$, and the output resolution $m$. The rule quantifies how offline learning effort and online discretization must be balanced in order to achieve optimal approximation performance. It is obtained from a coupled error analysis based on interpreting the online reconstruction as recovery from perturbed data, which yields a decomposition into learning and reconstruction errors that can be analyzed independently. As a consequence, we recover the oracle convergence rate
$m^{-(\sigma-\tau)/d}$ whenever $N$ scales appropriately with $m$ and $n$. The theoretical
predictions are validated on the Darcy flow benchmark.

As a second contribution, we introduced a physics-informed extension of the online reconstruction stage. By augmenting the reconstruction functional with a soft PDE collocation penalty, the governing equation can be enforced at evaluation time for each new test input without retraining and without additional PDE solves. The resulting physics-informed Tikhonov functional admits a closed-form representer theorem. Numerical experiments for the Poisson equation demonstrate consistent reductions in reconstruction error and indicate that $m_{\cL}=m$ collocation points provide an effective accuracy-cost tradeoff.

Several theoretical questions remain open. Most notably, a convergence-rate analysis for the physics-informed surrogate would provide a rigorous characterization of the interplay between the training budget $N$, the number of input observations $n$, the reconstruction resolution $m$, and the number of collocation points $m_{\cL}$. Such a result could lead to a physics-informed budget allocation rule and extend the present analysis beyond the purely data-driven setting

\end{document}